\documentclass[12pt]{article}
\usepackage{amsfonts}
\usepackage[T1]{fontenc}
\usepackage[utf8]{inputenc}
\usepackage[letterpaper]{geometry}
\usepackage{amsmath}
\usepackage{amssymb}
\usepackage{amsthm}
\usepackage{graphicx}
\usepackage[active]{srcltx}
\numberwithin{equation}{section}

\usepackage{amsmath}
\usepackage{graphicx}
\usepackage[active]{srcltx}
\usepackage{amssymb}
\usepackage{color}
\usepackage{amsmath}
\usepackage{graphicx}
\usepackage[active]{srcltx}
\include{srctex}

\newtheorem{theorem}{Theorem}[section]

\newtheorem{definition}[theorem]{Definition}

\newtheorem{lemma}[theorem]{Lemma}

\newtheorem{proposition}[theorem]{Proposition}
\newtheorem{remark}[theorem]{Remark}

\newtheorem*{hypothesis}{Hypothesis}

\def\E{\mathbb{E}}
\def\L{\mathcal{L}}
\def\R{\mathbb{R}}
\def\F{\mathcal{F}}
\def\H{\mathcal{H}}
\begin{document}

\title{On a class of stochastic partial differential equations}
\author{Jian Song}
\date{}
\maketitle
\begin{abstract}  This paper concerns the stochastic partial differential equation with multiplicative noise $\frac{\partial u}{\partial t} =\mathcal L u+u\dot W$, where $\mathcal L$ is the generator of a symmetric L\'evy process $X$, $\dot W$ is a Gaussian noise and $u\dot W$ is understood both in the senses of Stratonovich and Skorohod.  The Feynman-Kac type of representations for the solutions and the moments of the solutions are obtained, and the H\"older continuity of the solutions is also studied. As a byproduct, when $\gamma(x)$ is a nonnegative and nonngetive-definite function, a sufficient and necessary condition for $\int_0^t\int_0^t |r-s|^{-\beta_0}\gamma(X_r-X_s)drds$ to be exponentially integrable is obtained. 
\end{abstract}

\section{Introduction}

In \cite{MR876085}, Walsh developed the theory of stochastic integrals with respect to martingale measures and used it to study the stochastic partial differential equations (SPDEs) driven by space-time Gaussian white noise. Dalang in his seminal paper \cite{MR1684157} extended the definition of Wash's stochastic integral and applied it to solve SPDEs with Gaussian noise white in time and homogeneously colored in space (white-colored noise). Recently, the theories on SPDEs with white-colored noise have been extensively developed, and one can refer to, for instance, \cite{MR2295103, MR1207136,MR1500166,  MR3222416, MR2329435} and the references therein. For the SPDEs with white-colored noise, the methods used in the above-mentioned literature relies on the martingale structure of the noise, and hence cannot be applied to the case when the noise is colored in time.  On the other hand, SPDEs driven by a Gaussian noise which is colored in time and (possibly) colored in space have attracted more and more attention. 

In the present article, we consider the following SPDE in $\mathbb R^d$,
\begin{equation}\label{spde}
\begin{cases}
\displaystyle\frac{\partial u}{\partial t}=\L u+u\dot {W},& t\ge 0, x\in \mathbb R^d\\
u(0,x)=u_0(x),& x\in\mathbb R^d.
\end{cases}
\end{equation}
In the above equation, $\L$ is the generator of a L\'evy process $\{X_t,t\ge 0\}$, $u_0(x)$ is a continuous and bounded function, and the noise $\dot{W}$ is a (generalized) Gaussian  random field independent of $X$ with the covariance function given by  
\begin{equation}\label{kernel}
\E[\dot{W}(t,x)\dot{W}(s,y)]=|t-s|^{-\beta_0}\gamma(x-y),
\end{equation}
where  $\beta_0\in(0,1)$ and $\gamma$ is a symmetric ,nonnegative and nonnegative-definite (generalized) function. The product $u\dot W$ in (\ref{spde}) is understood either in the {\it Stratonovich} sense 
or in the {\it Skorohod} sense. 
 Throughout the paper, we assume that $X$ is a symmetric L\'evy process with characteristic exponent $\Psi(\xi)$, i.e., $\E\exp(i\xi X_t)=\exp(-t\Psi(\xi))$. Note that the symmetry implies that $\Psi(\xi)$ is a real-valued nonnegative function. Furthermore, we assume that $X$ has transition functions denoted by $q_t(x)$, which also entails that $\lim_{|\xi|\to \infty} \Psi(\xi)=\infty$ by Riemann-Lebesgue lemma.

  When $\L=\frac12 \Delta$ where $\Delta$ is  the Laplacian operator, and $\dot W$ is colored in time and white in space, Hu and Nualart \cite{MR2449130} investigated the conditions to obtain a unique mild solution for (\ref{spde}) in the Skorohod sense, and obtained the Feynman-Kac formula for the moments of the solution. When $\L =\frac12 \Delta$,   and $\dot W$ is a fractional white noise with Hurst parameters $H_0\in (\frac12, 1)$ in time and $(H_1,\dots, H_d)\in (\frac12, 1)^d$ in space, i.e., $\beta_0= 2-2H_0$ and $\gamma(x)=\prod_{i=1}^d|x_i|^{2H_i-2}$,  Hu et al. \cite{MR2778803}  obtained a Feynman-Kac formula for a weak solution under the condition $2H_0+\sum_{i=1}^d H_i>d+1$ for the SPDE in the Stratonovich sense. This result was extended to the case $\L=-(-\Delta)^{\alpha/2}$ in Chen et al. \cite{CHS}. A recent paper  \cite{HHNT}  by Hu et al. studied (\ref{spde}) in both senses when $\L=\frac12\Delta$ and $\dot W$ is a general Gaussian noise,  obtained the Feynman-Kac formulas for the solutions and the moments of the solutions,  and investigated H\"older continuity of the Feynman-Kac functional  and the intermittency of the solutions.

There has been fruitful literature on (\ref{spde}) in the sense of Skorohod, especially when $\dot W$ is white in time. For instance, when $\L=\frac12\Delta$, (\ref{spde})  is the well-known {\it parabolic Anderson model}  (\cite{Anderson}) and has been extensively investigated in, for example, \cite{MR1316109, MR1185878, MR1111743}.  Foondun and Khoshnevisan \cite{ MR2480553, MR2984063} studied the general nonlinear SPDEs.  For SPDE \eqref{spde} with space-time colored noise, the intermittency property of the solution was investigated in \cite{Balan-Conus-2013, MR3414457}  when $\L=\frac12\Delta$,  and in \cite{MR3262944} when $\L=-(-\Delta)^{\alpha/2}$.

The main purpose of the current paper is to study (\ref{spde}) in both senses of Stratonovich and Skorohod under the assumptions Hypothesis (I) in Section 3 and Hypothesis (II) in Section 5.1 respectively. Under Hypothesis (I),  we will obtain Feynman-Kac type of representations for a mild solution  to  (\ref{spde}) in the  Stratonovich sense and for the moments of the solution (Theorem \ref{thmfk} and Theorem \ref{thmfkmom}). Under Hypothesis (II), we will show that the mild solution to  (\ref{spde}) in the Skorohod sense exists uniquely, and obtain the Feynman-Kac formula for the moments of the solution (Theorem \ref{thmfkskr} and Theorem \ref{thmfkmom'}). Furthermore, under stronger conditions, we can get H\"older continuity of the solutions in both senses (Theorem \ref{thmholder} and Theorem \ref{thmholder'}). As a byproduct, we show that Hypothesis (I) is  a sufficient and necessary condition for the Hamiltonian $\int_0^t\int_0^t |r-s|^{-\beta_0}\gamma(X_r-X_s)drds$ to be exponentially integrable (Proposition \ref{integrable} and Theorem \ref{expint}).

There are two key ingredients to prove the main result Theorem \ref{thmfk} for the Stratonovich case. One is to obtain the exponential integrability of $\int_0^t\int_0^t |r-s|^{-\beta_0}\gamma(X_r-X_s)drds$. When $X$ is a Brownian motion, Le Gall's moment method (\cite{MR1329112}) was applied in \cite{MR2778803} to get the exponential integrability, and when $X$ is a symmetric $\alpha$-stable process,  the techniques from large deviation were employed in \cite{CHS, MR3414457}. However, in the current paper, we cannot apply directly either of the two approaches  due to the lacks of  the self-similarity of the L\'evy process $X$ and the homogeneity of the spatial kernel function $\gamma(x)$.   Instead, to get the desired exponential integrability, we estimate the moments of $\int_0^t\int_0^t |r-s|^{-\beta_0}\gamma(X_r-X_s)drds$ directly using Fourier analysis inspired by \cite{HHNT} and the techniques for the computation of moments used in \cite{MR2473265}.   The other key ingredient  is to justify that the Feynman-Kac representation (\ref{mildu}) is a mild solution to (\ref{spde}) in the sense of Definition \ref{defmildstr}. To this goal, we will apply the Malliavin calculus and follow the ``standard'' approach used in \cite{MR2778803, CHS, HHNT}.  

We get the existence of the solution to (\ref{spde}) in the Stratonovich sense by finding its Feynman-Kac representation directly, while in this article we do not address its uniqueness which will be our future work. A possible ``probabilistic'' treatment that was used in \cite{MR2544739} is to express the Duhamel solution as a sum of multiple Stratonovich integrals, and then investigate its relationship (the Hu-Meyer formula \cite{MR960509}) with the Wiener chaos expansion. Another approach is to consider (\ref{spde}) pathwisely as  a ``deterministic'' equation. Hu et al. \cite{HHNT} obtained the existence and uniqueness of (\ref{spde}) in the Stratonovich sense when $\L=\frac12 \Delta$ and $\dot W$  is a general Gaussian noise, by linking it to a general pathwise equation for which the authors obtained the existence and uniqueness in the framework of weighted Besov spaces. For general SPDEs, one can refer to \cite{MR2765508, MR2925571,MR2255351, MR2599193} for the rough path treatment. Recently, Deya \cite{Deya} applied Hairer's regularity structures theory (\cite{MR3071506}) to investigate a nonlinear heat equation driven by a space-time fractional white noise.

For (\ref{spde}) in the Skorohod sense, we obtain the existence and uniqueness result by studying the chaos expansion of the solution as has been done in \cite{MR2449130, MR3262944, HHNT}. We apply the approximation method initiated in \cite{MR2449130} to get the Feynman-Kac type of representation for the moments of the solution.  One possibly can also obtain the representation by directly computing the expectations of the products of Wiener chaoses as in \cite{MR3080991}. 

The rest of the paper is organized as follows. In Section \ref{sectionpre}, we recall some preliminaries on the Gaussian noise and Malliavin calculus. In Section \ref{sectionei}, we provide a sufficient and necessary condition for the Hamiltonian $\int_0^t\int_0^t |r-s|^{-\beta_0}\gamma(X_r-X_s)drds$ to be exponentially integrable.  In Section \ref{sectionstr}, the Feynman-Kac formula for a mild solution to (\ref{spde}) in the Stratonovich sense is obtained, the Feynman-Kac formula for the moments of the solution is provided, and  the H\"older continuity of the solution is studied. Finally, in Section \ref{sectionsk}, we obtain the existence and uniqueness of the mild solution in the Skorohod sense under some condition, find the Feynman-Kac formula for the moments, and investigate the H\"older continuity of the solution.

\section{Preliminaries}\label{sectionpre}
In this section, we introduce the stochastic integral with respect to the noise $\dot W$ and recall some material from Malliavin calculus which will be used. 
 
Let $C_0^{\infty}(\mathbb R_+\times \mathbb R^d)$ be the space of smooth functions on $\mathbb R_+\times \mathbb R^d$ with compact supports, and the Hilbert space $\mathcal H$ be the completion of $C_0^\infty(\mathbb R_+\times \mathbb R^d)$ endowed with the inner product
 \begin{equation}\label{inner}
\langle \varphi, \psi\rangle_{\mathcal H}=\int_{\mathbb R_+^2}\int_{\mathbb R^{2d}} \varphi(s,x)\psi(t,y) |t-s|^{-\beta_0}\gamma(x-y)dsdtdxdy,
\end{equation}
where  $\beta_0\in(0,1)$ and $\gamma$ is a symmetric, nonnegative and nonnegative-definite function. Note that $\mathcal H$ contains all measurable functions $\phi$ satisfying
$$\int_{\mathbb R_+^2}\int_{\mathbb R^{2d}} |\phi(s,x)||\phi(t,y)| |t-s|^{-\beta_0}\gamma(x-y)dsdtdxdy<\infty.$$
In a complete probability space $(\Omega, \mathcal F, P)$,  we define an isonormal Gaussian process (see, e.g., \cite[Definition 1.1.1]{MR2200233}) $W=\{W(h), h\in \mathcal H \}$  with the covariance function given by
$\E[W(\varphi)W(\psi)]=\langle \varphi, \psi\rangle_\mathcal H.$
In this paper, we will also use the following stochastic integral to denote $W(\varphi)$,
\[W(\varphi):=\int_0^\infty \int_{\R^d} \varphi(s,x) W(ds,dx).\]

Denote $\mathcal S(\R^d)$ the Schwartz space of  rapidly decreasing functions and let $\mathcal S'(\R^d)$ denote its dual space of tempered distributions.  Let $\widehat \varphi$ or $\mathcal F\varphi$ be the Fourier transform of $\varphi\in \mathcal S'(\R^d)$, which can be defined as the following integral if $\varphi \in L^1(\R^d)$,
$$\widehat \varphi(\xi)=\mathcal F \varphi(\xi):=\int_{\R^d} e^{-i\xi\cdot x}\varphi(x)dx.$$
By the Bochner-Schwartz theorem (see, e.g., Theorem 3 in Section 3.3, Chapter II in \cite{MR0435834}), the {\it spectral measure} $\mu$ of the process $W$ defined by 
\begin{equation}\label{spectral}
\int_{\mathbb R^d} \gamma(x) \varphi(x) dx=\frac{1}{(2\pi)^d}\int_{\mathbb R^d} \widehat \varphi(\xi)\mu(d\xi), \quad \forall\, \varphi\in\mathcal S(\mathbb R^d)
\end{equation}
 exists and is positive and tempered (meaning that there exists $p\ge 1$ such that $\int_{\R^d}(1+|\xi|^2)^{-p} \mu(d\xi)<\infty$). The inner product in (\ref{inner}) now can be represented by: 
\begin{equation}\label{inner'}
\langle \varphi, \psi\rangle_{\mathcal H}= \frac{1}{(2\pi)^d}
\int_{\mathbb R_+^2}\int_{\mathbb R^{d}} \widehat \varphi(s,\xi)\overline{\widehat \psi(t,\xi)} |t-s|^{-\beta_0}\mu(d\xi)dsdt,\quad \forall\, \varphi, \psi\in C_0^\infty(\R_+\times\R^d),
\end{equation}
where the Fourier transform is with respect to the space variable only, and $\overline z$ is the complex conjugate of $z$.

Throughout the paper, we assume that  the symmetric covariance function $\gamma(x)$ possesses the following properties . 
\begin{itemize}
\item[(1)] $\gamma(x)$ is nonnegative and locally integrable. 
\item[(2)] Its Fourier transform $\widehat \gamma(\xi) \in \mathcal S'(\R^d)$ is a measurable function which is nonnegative almost everywhere. 
\item[(3)] $\gamma(x)$ is a continuous functions mapping from $\mathbb\R^d$ to $[0,\infty]$, where $[0,\infty]$ is the usual one-point compactification of $[0,\infty)$.
\item[(4)] $\gamma(x)<\infty$ if and only if $x\neq0$ ~{\bf OR} ~$\widehat \gamma\in L^\infty(\R^d)$ and $\gamma(x)<\infty$ when $x\neq 0$ .
\end{itemize}
Note that  the function $\widehat \gamma$ is a tempered distribution, and hence it is also locally integrable. Consequently, the spectral measure $\mu(d\xi)=\widehat\gamma(\xi)d\xi$ is absolutely continuous with respect to the Lebesgue measure.  The function $\gamma(x)$ with the above four properties covers a   number of kernels such as the Riesz kernel $|x|^{-\beta}$ with $\beta\in(0,d)$,  the Cauchy Kernel $\prod_{j=1}^d(x_j^2+c)^{-1}$, the Poisson kernel $(|x|^2+c)^{-(d+1)/2}$, 
the Ornstein-Uhlenbeck kernel $e^{-c|x|^{\alpha}}$ with $\alpha\in (0,2]$, and the kernel  of fractional white noise $\prod_{j=1}^d|x_j|^{-\beta_j}$ with $\beta_j\in(0,1), j=1,\dots, d$, where the generic constant $c$ is a positive number.

For Borel probability measures $\nu_1(dx)$ and $\nu_2(dx)$,  the {\it mutual energy between $\nu_1$ and $\nu_2$ in gauge $\gamma$} is defines as follows (\cite{MR2529437})
$$\mathcal E_\gamma(\nu_1, \nu_2):=\int_{\R^d}\int_{\R^d} \gamma(x-y) \nu_1(dx)\nu_2(dy).$$
When $\nu_1=\nu_2=\nu$, we denote $\mathcal E_{\gamma}(\nu):=\mathcal E_{\gamma}(\nu, \nu)$ and it is called the {\it $\gamma$-energy of the measure $\nu$}. When both $\mathcal E_{\gamma}(\nu_1)$ and $\mathcal E_{\gamma}(\nu_2)$ are finite, by \cite[Lemma 5.6, or Equation (5.37)]{MR2529437},  the following identity holds,
\begin{equation}\label{exptransform'}
\int_{\R^d}\int_{\R^d} \gamma(x-y) \nu_1(dx)\nu_2(dy)=\frac1{(2\pi)^d} \int_{\R^d} \widehat\gamma(\xi)\mathcal F\nu_1(\xi)\overline{\mathcal F\nu_2(\xi)}d\xi,
\end{equation}
where, for a Borel probability measure $\lambda(dx)$, $\mathcal F \lambda(\xi):=\int_{\R^d} e^{-i\xi\cdot x}\lambda(dx)$ is its Fourier transform.

%

A function is called a \emph{kernel of positive type} (\cite[Definition 5.1]{MR2529437}) if it satisfies properties (1) and (2). For kernels of positive type, we have Parseval's Formula and the maximum principle as stated in the following lemma.
\begin{lemma}\label{lemma}
Let $g$ and $f$ be kernels of positive type. Assume that $g\in L^1(\R^d)$. Then if $\int_{\R^d} g(x) f(x) dx<\infty$ or $\int_{\R^d} \widehat g(\xi) \widehat f(\xi)d\xi<\infty$, we have 
\begin{equation}\label{equ2.5'}
\int_{\R^d} g(x) f(x) dx= \int_{\R^d} \widehat g(\xi) \widehat f(\xi)d\xi.
\end{equation}
Furthermore, the following maximum principle holds,
\begin{equation}\label{equ2.6'}
\int_{\R^d} g(x+a) f(x) dx\le \int_{\R^d} g(x) f(x) dx, \,\,\int_{\R^d} \widehat g(\xi+\eta) \widehat f(\xi)d\xi \le \int_{\R^d} \widehat g(\xi) \widehat f(\xi)d\xi, 
\end{equation}
for all $a, \eta \in\R^d.$
\end{lemma}
\begin{proof}
Without loss of generality, we assume $\int_{\R^d} \widehat g(\xi) \widehat f(\xi) d\xi<\infty. $ Let $p(x)\in C_0^\infty(\R^d)$ be a symmetric probability density function such that $\widehat p(\xi)\ge 0$ for all $\xi\in \R^d$. Denote  $p_\varepsilon(x) =\frac{1}{\varepsilon^d}p(\frac{x}\varepsilon)$.   Note that $g\in L^1(\R^d)$ implies that $\widehat g$ is continuous everywhere. Consequently, $p_\varepsilon*g$ converges to $g$ in $L^1$ and almost everywhere, and $p_\varepsilon*\widehat g$ converges to $\widehat g$ everywhere.  For any fixed $a\in\R^d$,
\begin{align}
&\int_{\R^d} g(x+a) f(x) dx =\int_{\R^d} \lim_{\varepsilon\to 0} (p_\varepsilon*g)(x+a) f(x) dx\notag\\
& \le  \lim_{\varepsilon\to 0}  \int_{\R^d} (p_\varepsilon*g)(x+a) f(x) dx \text{ (Fatou's lemma)} \notag\\
&= \lim_{\varepsilon\to 0}  \int_{\R^d} \widehat p_\varepsilon(\xi) \widehat g(\xi) e^{i\xi\cdot a} \widehat f(\xi) d\xi \text{ (Parseval's Formula)}\notag\\
&\le \lim_{\varepsilon\to 0}  \int_{\R^d} \widehat p_\varepsilon(\xi)\widehat g(\xi) \widehat f(\xi) d\xi= \int_{\R^d} \widehat g(\xi) \widehat f(\xi) d\xi \text{ (dominated convergence theorem)}.\label{equ2.5}
\end{align}
In particular, we have 
\begin{equation}\label{equ2.6}
\int_{\R^d} g(x) f(x) dx\le  \int_{\R^d} \widehat g(\xi) \widehat f(\xi)d\xi<\infty. 
\end{equation}
Therefore, in the same way of obtaining \eqref{equ2.5}, we can show that, for any fixed $\eta\in \R^d$,
\begin{equation}\label{equ2.7}
\int_{\R^d} \widehat g(\xi+\eta) \widehat f(\xi) d\xi\le \int_{\R^d} g(x) f(x) dx, 
\end{equation}
and especially we have,
\begin{equation}\label{equ2.8}
\int_{\R^d} \widehat g(\xi) \widehat f(\xi) d\xi\le \int_{\R^d} g(x) f(x) dx. 
\end{equation}
Hence under the assumption $\int_{\R^d} \widehat g(\xi) \widehat f(\xi) d\xi<\infty$, the Parseval's identity \eqref{equ2.5'} follows from \eqref{equ2.6} and \eqref{equ2.8}, and the maximum principle \eqref{equ2.6'} follows from \eqref{equ2.5}, \eqref{equ2.7} and \eqref{equ2.5'}.
\hfill 
\end{proof}
\noindent This allows us to have the following result for the computation of $\E[\gamma(X_t)]$. Recall that $q_t(x)$ is the transition density function of the L\'evy process $X$ and $\widehat q_t(\xi)=\E\exp(i\xi X_t)=\exp(-t\Psi(\xi)).$
\begin{lemma}\label{expectation}
If $\int_{\R^d} e^{-t\Psi(\xi)}\mu(d\xi)<\infty$ or $\E[\gamma(X_t)]<\infty$,  we have
$$\E[\gamma(X_t)]=\int_{\R^d}\gamma(x) q_t(x) dx=\frac1{(2\pi)^d}\int_{\R^d} e^{-t\Psi(\xi)}\mu(d\xi),$$
and $$\E[\gamma(X_t+a)]\le \E[\gamma(X_t)], \, \forall a\in \R^d.$$
\end{lemma}
\noindent Another consequence of Lemma \ref{lemma} is the following property of the spectral measure $\mu.$
\begin{lemma}\label{lem2.3measure}
For any bounded set $A\in\mathcal B(\R^d)$, 
$$\sup_{z\in \R^d}\mu([\xi+z \in A])<\infty.$$
\end{lemma}
\begin{proof}
 Since $A$ is bounded, there exists a positive constant $C$ such that $I_A(\xi) \le C \exp(-|\xi|^2)$ for all $\xi\in \R^d$. Therefore
\begin{align*}
\mu([\xi+z \in A])=\int_{\R^d}  I_A(\xi+z) \mu(d\xi) \le C \int_{\R^d}  \exp(-|\xi+z|^2) \mu(d\xi) \le C \int_{\R^d}  \exp(-|\xi|^2) \mu(d\xi),
\end{align*}
where the last step follows form Lemma \ref{lemma}, noting that $\exp(-|\xi|^2)\in L^1(\R^d)$ is a kernel of positive type. Then the result follows from the fact that $\int_{\R^d}  \exp(-|\xi|^2) \mu(d\xi)<\infty$.
\end{proof}

Now we briefly recall some useful information in Malliavin calculus.  The reader is referred to \cite{MR2200233} for more details. Let $D$ be the Malliavin derivative, which is an operator mapping from the Sobolev space $\mathbb D^{1,2}\subset L^2(\Omega)$  endowed with the norm$\|F\|_{1,2}=\sqrt{\E[F^2]+\E[\|DF\|_\mathcal H^2]}$
to $L^2(\Omega; \mathcal H)$. The divergence operator $\delta$  is defined as the  the dual operator of $D$ by the duality $\E[F\delta(u)]=\E[\langle DF, u\rangle_\mathcal H]$
for all $F\in \mathbb D^{1,2}$ and $u\in L^2(\Omega;\mathcal H)$ in the domain of $\delta$.  Note that when $u\in \mathcal H,  \delta(u)=W(u)$, and that the operator $\delta$ is also called the \emph{Skorohod integral} since it coincides with the Skorohod integral in the case of Brownian motion. When $F\in \mathbb D^{1,2}$ and $h\in \mathcal H$, we have
\begin{equation}\label{wick}
\delta(Fh)=F\diamond \delta(h),
\end{equation} 
where $\diamond$ means the Wick product.  For $u$ in the domain of $\delta$, we also denote $\delta(u)$ by $\int_0^\infty\int_{\R^d} u(s,y) W^\diamond(ds,dy)$ in this article. 
The following two formulas will be used in the proofs. 
\begin{equation}\label{intbypart'}
FW(h)=\delta(Fh)+\langle DF, h\rangle_\mathcal H,
\end{equation}
for all $F\in \mathbb D^{1,2}$ and  $h\in \mathcal H$. 
\begin{equation}\label{intbypart}
\E\left[FW(h)W(g)\right]=\E\left[\langle D^2F, h\otimes g\rangle_{\H^{\otimes 2}}\right]+\E[F]\langle h, g\rangle_{\H},
\end{equation}
for all $F\in \mathbb D^{2,2}, h\in \mathcal H, g\in \mathcal H.$

The Wiener chaos expansion has been used in, e.g.,  \cite{MR2449130,Balan-Conus-2013}, to deal with (\ref{spde}) in the Skorohod sense. Here we recall some basic facts. Let $F$ be a square integrable random variable measurable with respect to the $\sigma$-algebra generated by $W$. Then $F$ has the chaos expansion
\[F=\E[F]+\sum_{n=1}^\infty F_n,\]
where $F_n$ belongs to the $n$-th Wiener chaos space $\mathbb H_n$. Moreover, $F_n=I_n(f_n)$ for some $f_n\in \mathcal H^{\otimes n}$, and the expansion is unique if we require that all $f_n$'s are symmetric in its $n$ variables.  Here $I_n: \mathcal H^{\otimes n}\to \mathbb H_n$ is the multiple Wiener integral. We have the following isometry
\begin{equation}\label{e2.10}
\E[|I_n(f_n)|^2]=n!\| \widetilde f_n\|^2_{\mathcal H^{\otimes n}},
\end{equation}
where $\widetilde f_n$ is the symmetrization of $f_n$.

\section{On the exponential integrability}\label{sectionei}

In this section, we will show that Hypothesis (I) below is a sufficient and necessary condition such that for all $\lambda, t>0$
\[\E\left[\exp\left(\lambda \int_0^t\int_0^t |r-s|^{-\beta_0}\gamma(X_r-X_s)drds\right)\right]<\infty.\]

\begin{hypothesis}[I]  \label{h1}The spectral measure $\mu$ satisfies
\[\int_{\R^d}\frac1{1+(\Psi(\xi))^{1-\beta_0}}\mu(d\xi)<\infty.\]
\end{hypothesis}
\begin{remark}
When $\L=-(-\Delta)^{\alpha/2}$ for $\alpha\in (0,2]$ and $\gamma(x)$ is $|x|^{-\beta}$ with $\beta\in(0,d)$ or $\prod_{j=1}^d|x_j|^{-\beta_j}$ with $\beta_j\in(0,1), j=1,\dots, d$, Hypothesis (I) is equivalent to $\beta<\alpha(1-\beta_0)$, denoting $\beta=\beta_1+\dots+\beta_d$.
\end{remark}      
\noindent The following proposition shows that Hypothesis (I) is a necessary condition.
\begin{proposition}\label{integrable}
  Hypothesis (I) is a sufficient and necessary condition such that
 \begin{equation}\label{integrability}
 \E\int_0^t\int_0^t |r-s|^{-\beta_0}\gamma(X_r-X_s)drds<\infty,\, \text{ for all $t>0$}.
 \end{equation}

\end{proposition}
\begin{proof}
Since $\Psi(\xi)$ goes to infinity as $|\xi|$ goes to infinity, Hypothesis (I) implies $\int_{\R^d} e^{-t\Psi(\xi)}\mu(d\xi)<\infty$ for all $t>0.$ Hence by Lemma \ref{expectation}, either Hypothesis (I) or inequality \eqref{integrability} holds,
\begin{align*} 
\int_0^t\int_0^t |r-s|^{-\beta_0}\E\left[\gamma(X_r-X_s)\right]drds=\frac{1}{(2\pi)^d}\int_0^t\int_0^t |r-s|^{-\beta_0}\int_{\mathbb R^d} e^{-|r-s|\Psi(\xi) }\mu(d\xi)drds,
\end{align*} 
and the result follows from Fubini's theorem and Lemma \ref{lemma0}. \hfill
\end{proof}

\noindent The sufficiency is provided by the following theorem which is the main result in this section.
\begin{theorem}\label{expint}

Let the measure $\mu$ satisfy Hypothesis (I), then for all $t,\lambda>0$,
$$\E\left[\exp\left(\lambda\int_0^t\int_0^t |r-s|^{-\beta_0} \gamma(X_r-X_s)dsdr\right)\right]<\infty.$$
\end{theorem}
\begin{remark}
The above theorem, together with Proposition \ref{integrable},   actually declares the equivalence between the integrability and the exponential integrability of $\int_0^t \int_0^t |r-s|^{-\beta_0}\gamma(X_r-X_s)dsdr$. This result is surprising since, for a general nonnegative random variable, its integrability does not imply its exponential integrability.  The equivalence in our situation is mainly a consequence of the Markovian property of the L\'evy process $X$. A result in the same flavor for $\int_0^t f(B_s)ds$ where $B$ is a standard Brownian motion and $f$ is a positive measurable function has been discovered  by Khasminskii \cite{MR0123373} (see, e.g., \cite[Lemma 2.1]{MR1717054}).
\end{remark}

\begin{proof} Note that $\int_0^t\int_0^t |r-s|^{-\beta_0} \gamma(X_r-X_s)dsdr=2\int_0^t\int_0^r |r-s|^{-\beta_0} \gamma(X_r-X_s)dsdr$, and equivalently we will study the exponential integrability of $\int_0^t\int_0^r |r-s|^{-\beta_0} \gamma(X_r-X_s)dsdr$. Inspired by the method in the proof of \cite[Theorem 1]{MR2473265}, we  estimate the $n$-th moments as follows.
\begin{align*}
&\E\left(\int_0^t \int_0^r  |r-s|^{-\beta_0} \gamma(X_r-X_s) dsdr\right)^n =\int_{[0< s< r< t]^n}\E\left(\prod_{j=1}^n|r_j-s_j|^{-\beta_0}\gamma(X_{r_j}-X_{s_j})\right)dsdr\\
=&n! \int_{[0< s< r< t]^n\cap [0<r_1< r_2\dots< r_n< t]}\E\left(\prod_{j=1}^n|r_j-s_j|^{-\beta_0}\gamma(X_{r_j}-X_{s_j})\right)dsdr\\
\le &n! \int_{[0< s< r<t]^n\cap [0< r_1< r_2\dots< r_n< t]}\prod_{j=1}^n|r_j-\eta_j|^{-\beta_0}\E\left[\gamma(X_{r_j}-X_{\eta_j})\right]dsdr.
\end{align*}
The last inequality, where $\eta_j$ is the point in the set $\{r_{j-1}, s_j, s_{j+1}, \dots, s_n\}$ which is  closest  to $r_j$ from the left, holds since $\E\left[\gamma(X_{r_j}-X_{s_j})\right]=\E\left[\gamma(X_{r_j}-X_{\eta_j}+X_{\eta_j}-X_{s_j})\right]\le \E\left[\gamma(X_{r_j}-X_{\eta_j})\right]$ by the independent increment property of $X$ and Lemma \ref{expectation}. Note that $dsdr$ actually means $ds_1\dots ds_n dr_1\dots dr_n$ in the above last three integrals. Throughout the article, we will take this kind of abuse of the notation for simpler exposition. 

Fix the points $r_1< \dots< r_n$, we can decompose the set $ [0< s<r< t]^n \cap [0< r_1< r_2\dots< r_n< t]$ into $(2n-1)!!$ disjoint subsets depending on which interval the $s_i$'s are placed in.  More precisely, $s_1$ must be in $(0, r_1)$, while $s_2$ could be in $(0, s_1)$, $(s_1, r_1)$ or $(r_1, r_2)$. Similarly, there are $(2j-1)$ choices to place $s_j$.  Over each subset, we denote the integral by 

$$I_\sigma:=\int_{[0< z_1< \dots< z_{2n}< t]} \prod_{j=1}^n|z_{\sigma(j)}-z_{\sigma(j)-1}|^{-\beta_0}\E\left[\gamma(X_{z_{\sigma(j)}}-X_{z_{\sigma(j)-1}})\right]dz,$$
where $\sigma_{(1)}<\dots<\sigma_{(n)}$ are $n$ distinct elements in the set $\{2,3,\dots, 2n\}$. Hence
\begin{equation}\label{nmoment}
\E\left[\left(\int_0^t\int_0^t |r-s|^{-\beta_0} \gamma(X_r-X_s)dsdr\right)^n\right]\le n! \times \bigg[\text{sum of the }  (2n-1)!! \text{ terms of } I_\sigma's \bigg].
\end{equation}

Next, for fixed $n$, we will provide a uniform upper bound for all $I_{\sigma}'s$. Noting that $X_{z_{\sigma(j)}}-X_{z_{\sigma(j)-1}}\overset{d}{=}X_{z_{\sigma(j)}-z_{\sigma(j)-1}}$ and letting $y_j=z_j-z_{j-1}$, we have
\begin{align}
I_\sigma=&\int_{[0< y_1+y_2+\dots+y_{2n}< t, \, 0< y_1,\cdots, y_{2n}< t]} \prod_{j=1}^n|y_{\sigma(j)}|^{-\beta_0}\E[\gamma(X_{y_{\sigma(j)}})]dy\notag\\
\le& \frac{t^n}{n!}\int_{[0< y_1+y_2+\dots+y_{n}< t,\, 0< y_1,\cdots, y_{n}< t]} \prod_{j=1}^n|y_i|^{-\beta_0}\E[\gamma(X_{y_j})]dy\notag\\
=& \frac{t^n}{n!}\int_{[0< z_1< \dots< z_{n}< t]} \prod_{j=1}^n|z_j-z_{j-1}|^{-\beta_0}\E[\gamma(X_{z_{j}-z_{j-1}})]dz\notag\\
= &\frac{t^n}{n!}\int_{[0< z_1< \dots< z_{n}< t]} \int_{\mathbb R^{nd}}\prod_{j=1}^n|z_j-z_{j-1}|^{-\beta_0} e^{-(z_j-z_{j-1})\Psi(\xi_j)}\mu(d\xi)dz. \label{I}
\end{align}
Note that
\begin{align}
&\int_{[0< z_1< \dots< z_{n}< t]} \int_{\mathbb R^{nd}}\prod_{j=1}^n|z_j-z_{j-1}|^{-\beta_0} e^{-(z_j-z_{j-1})\Psi(\xi_j)}\mu(d\xi)dz\notag\\
=&\int_{\Omega_{t}^{n}} \int_{\mathbb R^{nd}}\prod_{j=1}^ns_j^{-\beta_0} e^{-s_j\Psi(\xi_j)}\mu(d\xi)ds,\label{e.3.3}
\end{align}
where 
\begin{equation}\label{omega}
\Omega_t^n=\left\{(s_1,\dots, s_n)\in [0,\infty)^n: \sum_{j=1}^n s_j\le t\right\}.
\end{equation}
For fixed large $N$, denote 
\begin{equation}\label{e3.5}
\varepsilon_N=\int_{[|\xi|\ge N] } \frac{1}{(\Psi(\xi))^{1-\beta_0}} \mu(d\xi), ~ \text{ and } ~ m_N=\mu([|\xi|\le N]).
\end{equation}
Thus, by (\ref{nmoment}), (\ref{I}), (\ref{e.3.3}) and Proposition \ref{hhnt}, we have 
\begin{align}
 &\E\left[\left(\lambda\int_0^t\int_0^t |r-s|^{-\beta_0} \gamma(X_r-X_s)dsdr\right)^n\right]\notag\\
 \le& (2n-1)!!\lambda^nt^n\sum_{k=0}^n \binom nk \frac{\left(\Gamma(1-\beta_0)t^{1-\beta_0}\right)^{k}}{\Gamma(k(1-\beta_0)+1)}m_N^{k}\left[A_0 \varepsilon_N\right]^{n-k}.\label{nmoment'}
\end{align}
Now, for fixed $t$ and $\lambda$, we can choose $N$ sufficiently large such that $4A_0 \lambda t\varepsilon_N<1.$ Consequently, 
\begin{align*}
 &\E\left[\exp\left(\lambda\int_0^t\int_0^t |r-s|^{-\beta_0} \gamma(X_r-X_s)dsdr\right)\right]\\
 \le &\sum_{n=0}^\infty \lambda^n t^n \frac{(2n-1)!!}{n!}\sum_{k=0}^n \binom nk \frac{\left(\Gamma(1-\beta_0)t^{1-\beta_0}\right)^{k}}{\Gamma(k(1-\beta_0)+1)}m_N^{k}\left[A_0\varepsilon_N\right]^{n-k}\\
 = &\sum_{k=0}^\infty \frac{\left(\Gamma(1-\beta_0)t^{1-\beta_0}\right)^{k}}{\Gamma(k(1-\beta_0)+1)}m_N^{k} \left[A_0 \varepsilon_N\right]^{-k}  \sum_{n=k}^\infty \lambda^n t^n \frac{(2n-1)!!}{n!}  \binom nk \left[A_0 \varepsilon_N\right]^{n}\\
\le &\sum_{k=0}^\infty \frac{\left(\Gamma(1-\beta_0)t^{1-\beta_0}\right)^{k}}{\Gamma(k(1-\beta_0)+1)}m_N^{k} \left[A_0 \varepsilon_N\right]^{-k}  \sum_{n=k}^\infty \left[4\lambda A_0 t\varepsilon_N\right]^{n}\\ 
=&\frac{1}{1-4\lambda A_0 t\varepsilon_N} \sum_{k=0}^\infty \frac{\left(\Gamma(1-\beta_0)t^{1-\beta_0}\right)^{k}}{\Gamma(k(1-\beta_0)+1)}\left(4\lambda t m_N\right)^{k} <\infty,
\end{align*}
where in the second inequality we used the estimate $\frac{(2n-1)!!}{n!}  \binom nk\le 2^n \cdot 2^n=4^n.$
The proof is concluded. 
\end{proof}

The following proposition, which plays a key role in this article, is a generalized version of Lemma 3.3 in \cite{HHNT}.
\begin{proposition}\label{hhnt}
For $\beta_0\in[0,1)$, assume  
\[\int_{\R^d} \frac{1}{1+(\Psi(\xi))^{1-\beta_0}}\mu(d\xi)<\infty.\]  Then there exists a positive constant $A_0$ depending on $\beta_0$ only such that for all $N>0$,
\[\int_{\Omega_t^n}\int_{\R^{nd}} \prod_{j=1}^nr_j^{-\beta_0} e^{-r_j\Psi(\xi_j)}\mu(d\xi)dr\le \sum_{k=0}^n \binom nk \frac{\left(\Gamma(1-\beta_0)t^{1-\beta_0}\right)^{k}}{\Gamma(k(1-\beta_0)+1)}m_N^{k}\left[A_0\varepsilon_N\right]^{n-k},\]
where $\varepsilon_N$ and $m_N$ are given by (\ref{e3.5}), and $\Omega_t^n$ is given by (\ref{omega}).
\end{proposition}
\begin{proof} The proof essentially follows the approach used in the proof of  \cite[Lemma 3.3]{HHNT}. 

First note that the assumption implies that $\lim\limits_{N\to\infty}\varepsilon_N=0$, and since $\mu(d\xi)$ is a tempered measure, then $m_N<\infty$ for all $N>0$. For a subset $S$ of $\{1,2,\dots, n\}$, we denote its complement by $S^c$, i.e., $S^c:=\{1,2, \dots, n\}\backslash S.$
\begin{align*}
&\int_{\R^{nd}}\int_{\Omega_t^n} \prod_{j=1}^n r_j^{-\beta_0} e^{-r_j\Psi(\xi_j)}dr\mu(d\xi)\\
=&\int_{\R^{nd}}\int_{\Omega_t^n} \prod_{j=1}^n r _j^{-\beta_0} e^{-r_j\Psi(\xi_j)}\left[I_{[|\xi_j|\le N]}+I_{[|\xi_j|> N]}\right]dr\mu(d\xi)\\
=&\sum_{S\subset\{1,2,\dots, n\}} \int_{\R^d}\int_{\Omega_t^n} \prod_{l\in S} r_l^{-\beta_0} e^{-r_l\Psi(\xi_l)}I_{[|\xi_l|\le N]} \prod_{j\in S^c} r _j^{-\beta_0} e^{-r_j\Psi(\xi_j)}I_{[|\xi_j|> N]}dr\mu(d\xi)\\
\le &\sum_{S\subset\{1,2,\dots, n\}} \int_{\R^d}\int_{\Omega_t^n} \prod_{l\in S} r_l^{-\beta_0} I_{[|\xi_l|\le N]} \prod_{j\in S^c} r _j^{-\beta_0} e^{-r_j\Psi(\xi_j)}I_{[|\xi_j|> N]}dr\mu(d\xi).
\end{align*}
Note that $\Omega_t^n\subset \Omega_t^S\times \Omega_t^{S^c},$
where $\Omega_t^I=\{(r_i, i\in I): r_i\ge0, \sum_{j\in I} r_i\le t\}$ for any $I\subset \{1,2,\dots, n\}$. Therefore,
\begin{align*}
&\int_{\R^{nd}}\int_{\Omega_t^n} \prod_{j=1}^n r_j^{-\beta_0} e^{-r_j\Psi(\xi_j)}dr\mu(d\xi)\\
\le &\sum_{S\subset\{1,2,\dots, n\}} \int_{\R^{nd}}\int_{\Omega_t^S\times \Omega_t^{S^c}} \prod_{l\in S} r_l^{-\beta_0} I_{[|\xi_l|\le N]} \prod_{j\in S^c} r _j^{-\beta_0} e^{-r_j\Psi(\xi_j)}I_{[|\xi_j|> N]}dr\mu(d\xi).
\end{align*}
By Lemma \ref{multii}, we have
$$\int_{\Omega_t^S} \prod_{l\in S} r_l^{-\beta_0} dr=\frac{\left(\Gamma(1-\beta_0)t^{1-\beta_0}\right)^{|S|}}{\Gamma(|S|(1-\beta_0)+1)}.$$
On the other hand, there exists $A_0>0$ depending on $\beta_0$ only such that
\begin{align*}
&\int_{\Omega_t^{S^c}}  \prod_{j\in S^c} r _j^{-\beta_0} e^{-r_j\Psi(\xi_j)}dr\le \int_{[0,t]^{|S^c|}}  \prod_{j\in S^c} r _j^{-\beta_0} e^{-r_j\Psi(\xi_j)}dr\\
&\le \prod_{j\in S^c}\int_0^t r^{-\beta_0} e^{-r\Psi(\xi_j)}dr\le  \prod_{j\in S^c} A_0  (\Psi(\xi_j))^{-1+\beta_0},
\end{align*}
where the last equality holds since $\int_0^t r^{-\beta_0} e^{-ar}dr=a^{-1+\beta_0}\int_0^{at} s^{-\beta_0}e^{-s}ds\le a^{-1+\beta_0} \int_0^\infty s^{-\beta_0}e^{-s}ds.$
Therefore,
\begin{align*}
&\int_{\R^{nd}}\int_{\Omega_t^n} \prod_{j=1}^n r_j^{-\beta_0} e^{-r_j\Psi(\xi_j)}dr\mu(d\xi)\\
\le &\sum_{S\subset\{1,2,\dots, n\}}\int_{\R^{nd}} \frac{\left(\Gamma(1-\beta_0)t^{1-\beta_0}\right)^{|S|}}{\Gamma(|S|(1-\beta_0)+1)}\prod_{l\in S}I_{[|\xi_l|\le N]}  \prod_{j\in S^c} A_0(\Psi(\xi_j))^{-1+\beta_0}I_{[|\xi_j|> N]} \mu(d\xi)\\
=& \sum_{S\subset\{1,2,\dots, n\}}\frac{\left(\Gamma(1-\beta_0)t^{1-\beta_0}\right)^{|S|}}{\Gamma(|S|(1-\beta_0)+1)}A_0^{|S^c|} m_N^{|S|}\varepsilon_N^{|S^c|}\\
=&\sum_{k=0}^n\binom nk \frac{\left(\Gamma(1-\beta_0)t^{1-\beta_0}\right)^{k}}{\Gamma(k(1-\beta_0)+1)}A_0^{n-k} m_N^{k}\varepsilon_N^{n-k},
\end{align*}
and the proof is concluded. \hfill
\end{proof} 

\begin{remark}\label{remark3.8}
If we assume the following stronger condition, 
\[\int_{\mathbb R^d} \frac{1}{1+(\Psi(\xi))^{1-\beta_0-\varepsilon_0}}\mu(d\xi)<\infty\]
for some $\varepsilon_0\in(0, 1-\beta_0)$,
we may prove that for all $\lambda, t>0$
\begin{equation}\label{e3.7}
\E\left[\lambda \exp\left(\left|\int_0^t\int_0^t |r-s|^{-\beta_0} \gamma(X_r-X_s)dsdr\right|^p\right)\right]<\infty, \, \mbox{ when } p<\frac{1}{1-\varepsilon_0},
\end{equation}
without involving Proposition \ref{hhnt}. An outline of the proof is as follows.

Now we estimate the integral over $\R^{nd}$  in the last term of (\ref{I}) first. By (\ref{I}) and Lemma \ref{lemma1}, there exists $C>0$ depending only on $1-\beta_0-\varepsilon_0$ and $\mu(d\xi)$, such that
\[I_\sigma\le C^n \frac{t^n}{n!}\int_{[0< z_1< z_2 \dots< z_{n}< t]} \prod_{j=1}^n|z_j-z_{j-1}|^{-\beta_0}\prod_{j=1}^n [1+(z_j-z_{j-1})^{-1+\beta_0+\epsilon}] dz.\]
Denote $\tau=(\tau_1,\dots, \tau_n)$ and $|\tau|=\sum_{j=1}^n \tau_j$.
Then  $$\prod_{j=1}^n [1+(z_j-z_{j-1})^{-1+\beta_0+\epsilon_0}] = \sum_{\tau\in\{0,1\}^n} \prod_{j=1}^n (z_j-z_{j-1})^{\tau_j (-1+\beta_0+\varepsilon_0)}= \sum_{\tau\in\{0,1\}^n} J_\tau=\sum_{m=0}^n\sum_{|\tau|=m} J_\tau.$$
When $|\tau|=m$ and $t\ge 1$, by Lemma \ref{multii}, we have 
\[\int_{[0< z_1< z_2 \dots< z_{n}< t]} \prod_{j=1}^n|z_j-z_{j-1}|^{-\beta_0}J_\tau dz\le \frac{C^n t^{m\varepsilon_0 +(n-m)(1-\beta_0)}}{\Gamma(m\varepsilon_0 +(n-m)(1-\beta_0)+1)}\le \frac{C^nt^{n(1-\beta_0)}}{\Gamma (n\varepsilon_0+1)}\, ,\]
noting that $\varepsilon_0<1-\beta_0$.

Note that there are $\binom{n}{m}$ $J_\tau$'s for $|\tau|=m$, and hence 
\begin{align}\label{isigma}
I_\sigma \le  C^n \frac{t^{n(2-\beta_0)}}{n!} \sum_{m=0}^n \binom{n}{m} \frac{1}{\Gamma (n\varepsilon_0+1)}\le C^n \frac{t^{n(2-\beta_0)}}{n!} (n+1) 2^n \frac{1}{(n\varepsilon_0/3)^{n\varepsilon_0}},
\end{align}
where in the last step we use the properties $\binom{n}{m}\le 2^n$ and $\Gamma(x+1)\ge (x/3)^x$.

Combining (\ref{nmoment}) and (\ref{isigma}), we have, for all $\lambda>0$ and  $t>0$, 
\[\E\left[\left(\int_0^t\int_0^t |r-s|^{-\beta_0} \gamma(X_r-X_s)dsdr\right)^{n}\right]\le \left(Ct^{2-\beta_0}\right)^n (n!)^{1-\varepsilon_0},\]
where $C>0$  depends on $\beta_0,\varepsilon_0$ and $\mu(d\xi)$, and then (\ref{e3.7}) follows.

 \end{remark}
\noindent In the rest of this section are some useful lemmas.
\begin{lemma} \label{lemma0}
There exist positive constants $C_1$ and $C_2$ depending on $\beta_0$ only such that 
\[\frac{1}{1+x^{1-\beta_0}}\int_0^t s^{-\beta_0}e^{-s}ds\le \int_0^t s^{-\beta_0}e^{-sx}ds\le \frac{1}{1+x^{1-\beta_0}}(C_1+C_2 t^{1-\beta_0}), \, \forall\, x>0.\]
Similarly, there exist positive constants $D_1$ and $D_2$ depending on $\beta_0$ only such that
\[\frac{2}{1+x^{1-\beta_0}} \int_0^t\int_0^s r^{-\beta_0} e^{-r}drds \le \int_0^t\int_0^t |r-s|^{-\beta_0}e^{-|r-s|x}drds\le \frac{2}{1+x^{1-\beta_0}} (D_1t+ D_2t^{2-\beta_0}),\, \forall\, x>0.\]
\end{lemma}
\begin{proof}
A change of variable implies that 
$$\int_0^t s^{-\beta_0}e^{-sx}ds =x^{\beta_0-1}\int_0^{tx} r^{-\beta_0}e^{-r}dr.$$
The first inequality is a consequence of the following observation. When $x\ge1$,
\begin{align*}
x^{\beta_0-1}\int_0^{t} r^{-\beta_0}e^{-r}dr \le x^{\beta_0-1}\int_0^{tx} r^{-\beta_0}e^{-r}dr\le x^{\beta_0-1}\int_0^{\infty} r^{-\beta_0}e^{-r}dr,
\end{align*}
and  when $0<x<1$,
$$\int_0^t s^{-\beta_0}e^{-s}ds\le \int_0^t s^{-\beta_0}e^{-sx}ds\le \int_0^t s^{-\beta_0}ds. $$
The second estimate follows from the first one and the following equality 
$$\int_0^t\int_0^t |r-s|^{-\beta_0}e^{-x|r-s|}dsdr= 2\int_0^{t}\int_0^{r}(r-s)^{-\beta_0}e^{-x(r-s)}dsdr=2\int_0^t \int_0^r s^{-\beta_0} e^{-xs}dsdr.\qedhere$$
\end{proof}
\begin{remark}\label{rm3.3}
 Using similar approach in the above proof, we can show that the two inequalities hold for $$\sup_{a\in\R}\int_0^t s^{-\beta_0} e^{-|s+a|x}ds ~~\text{ and }~~ \sup_{a\in\R}\int_0^t\int_0^t |r-s|^{-\beta_0}e^{-|r-s+a|x}drds$$ as well. It suffices to show that the upper bounds hold. We prove the first one as an illustration.
 When $0<x<1$, 
\[ \int_0^t s^{-\beta_0} e^{-|s+a|x}ds\le  \int_0^ts^{-\beta_0}ds;\]
 when $x\ge1$,
 $$\int_0^t s^{-\beta_0} e^{-|s+a|x}ds\le x^{\beta_0-1}\int_0^\infty s^{-\beta_0}e^{-|s+ax|}ds\le Cx^{\beta_0-1}$$ where   $$C=\sup_{a\in\R}\int_0^\infty s^{-\beta_0}e^{-|s+ax|}ds\le\int_0^1 s^{-\beta_0}ds+\sup_{a\in\R}\int_1^\infty e^{-|s+ax|}ds\le \int_0^1 s^{-\beta_0}ds+\int_{-\infty}^\infty e^{-|s|}ds<\infty,$$
 and the upper bound is obtained.
\end{remark}

%
%

\begin{lemma}\label{lemma1}
Suppose $$\int_{\mathbb R^d} \frac{1}{1+(\Psi(\xi))^{\alpha}}\mu(d\xi)<\infty, $$ for some $\alpha>0$, then  there exists a constant $C>0$ depending on $\mu(d\xi)$ and $\alpha$ only, such that 
\[\int_{\mathbb R^d} e^{-x\Psi(\xi)} \mu(d\xi)\le C(1+x^{-\alpha}),\, \forall x>0.\]
\end{lemma}
\begin{proof} Since $\lim_{|\xi|\to \infty} \Psi(\xi)=\infty,$ we can choose $M>0$ such that $\Psi(\xi)>1$ when $|\xi|>M$. Clearly 
\[\int_{\mathbb R^d} e^{-x\Psi(\xi)} \mu(d\xi)=\int_{[|\xi|\le M|]} e^{-x\Psi(\xi)} \mu(d\xi)+\int_{[|\xi|>M]} e^{-x\Psi(\xi)} \mu(d\xi).\]
The first integral on the right-hand side is bounded by $\mu([|\xi|\le M])$ which is finite. For the second integral, note that $y^{\alpha}e^{-y}$ is uniformly bounded for all $y\ge0$, and hence there exists a constant $C$ depending on $\alpha$ only such that
\[\int_{[|\xi|>M]} e^{-x\Psi(\xi)} \mu(d\xi)\le C \int_{[|\xi|>M]} x^{-\alpha} (\Psi(\xi))^{-\alpha}\mu(d\xi)\le  x^{-\alpha} \int_{[|\xi|>M]}  \frac{2C}{1+(\Psi(\xi))^{\alpha}}\mu(d\xi).\qedhere\]
\end{proof}

\begin{lemma}\label{multii}
Suppose $\alpha_i\in (-1, 1), i=1, \dots, n$ and let $\alpha=\alpha_1+\dots+\alpha_n$.
Then \[\int_{[0< r_1<\dots<r_n<t]}~ \prod_{i=1}^{n}(r_i-r_{i-1})^{\alpha_i}~dr_1\dots dr_n= \frac{\prod_{i=1}^n\Gamma(\alpha_i+1)t^{\alpha+n}}{\Gamma(\alpha+n+1)},\]
where $\Gamma(x)=\int_0^\infty t^{x-1} e^{-t}dt$ is the Gamma function.
\end{lemma}
\begin{proof}
The result follows from a direct computation of the iterated integral with respect to $r_n, r_{n-1}, \dots, r_1$ orderly. The properties  $\Gamma(x+1)=x\Gamma(x)$ and $B(x,y)=\frac{\Gamma(x)\Gamma(y)}{\Gamma(x+y)}$ are used in the computation, where 
 $B(x,y) := \int_0^1 t^{x-1} (1-t)^{y-1} dt $ for $x,y>0$  is the beta function.
\end{proof}

\section{Stratonovich equation}\label{sectionstr}

In this section, we will use the approximation method (\cite{MR2778803, CHS, HHNT}) to study (\ref{spde}) in the Stratonovich sense. 

\subsection{Definition of $ \int_0^t\int_{\R^d} \delta_0(X_{t-r}^x-y)W(dr,dy)$}

Denote $g_\delta(t):=\frac1\delta I_{[0,\delta]}(t)$ for $t\ge 0$ and $p_\varepsilon(x)=\frac1{\varepsilon^{d}}p(\frac{x}{\varepsilon})$ for $x\in \R^d$, where $p(x)\in C_0^{\infty}(\R^d)$ is a symmetric probability density function and its Fourier transform $\widehat p(\xi)\ge 0$  for all $\xi\in\R^d.$ We also have that for all $\xi\in\R^d, \lim_{\varepsilon\to 0+}\widehat p_\varepsilon(\xi) =1.$ 

Let 
\begin{equation}\label{Phi}
\Phi_{t,x}^{\varepsilon, \delta}(r,y):=\int_0^t g_{\delta}(t-s-r)p_\varepsilon(X_s^x-y)ds\cdot I_{[0,t]}(r).
\end{equation}
Formal computations suggest that 
\[\lim_{\varepsilon,\delta\downarrow0}\int_0^t\int_{\R^d}\Phi_{t,x}^{\varepsilon, \delta}(r,y) W(dr,dy) = \int_0^t\int_{\R^d} \delta_0(X_{t-r}^x-y)W(dr,dy), \]
where $\delta_0(x)$ is the Dirac delta function. This formal derivation is validated by the following theorem.
\begin{theorem}\label{appfeyn}
Let the measure $\mu$ satisfy  Hypothesis (I),
then $W(\Phi_{t,x}^{\varepsilon, \delta})$ is well-defined a.s. and  forms a Cauchy sequence in $L^2$ when $(\varepsilon,\delta)\to 0$ with the limit denoted by
\[W(\delta_0(X_{t-\cdot}^x-\cdot)I_{[0,t]}(\cdot))=\int_0^t\int_{\R^d} \delta_0(X_{t-r}^x-y)W(dr,dy).\]
Furthermore, $W(\delta_0(X_{t-\cdot}^x-\cdot)I_{[0,t]}(\cdot))$ is Gaussian distributed conditional on $X$ with variance 
\begin{equation}\label{cc}
\mathrm{Var}\left[W(\delta_0(X_{t-\cdot}^x-\cdot)I_{[0,t]}(\cdot))|X\right]=\int_0^t\int_0^t |r-s|^{-\beta_0} \gamma(X_r-X_s)dsdr.
\end{equation}
\end{theorem}
\begin{proof}
Let $\varepsilon_i, \delta_i, i=1,2$ be positive numbers, 
then by (\ref{inner}) 
\begin{align}\label{e4.3}
\langle \Phi_{t,x}^{\varepsilon_1,\delta_1}, \Phi_{t,x}^{\varepsilon_2,\delta_2}\rangle_\mathcal H=&\int_{[0,t]^4} \int_{\R^{2d}} p_{\varepsilon_1}(X_{s_1}^x-y_1)p_{\varepsilon_2}(X_{s_2}^x-y_2)\gamma(y_1-y_2)\notag\\
&\qquad g_{\delta_1}(t-s_1-r_1)g_{\delta_2}(t-s_2-r_2)
|r_1-r_2|^{-\beta_0}dy_1dy_2dr_1dr_2ds_1ds_2 .
\end{align}
Hence $$\langle \Phi_{t,x}^{\varepsilon_1,\delta_1}, \Phi_{t,x}^{\varepsilon_2,\delta_2}\rangle_\mathcal H\ge0.$$ 
By \cite[Lemma A.3]{MR2778803}, there exists a positive constant $C$ depending on $\beta_0$ only, such that
\begin{equation}\label{gdelta}
\int_{[0,t]^2} g_{\delta_1}(t-s_1-r_1)g_{\delta_2}(t-s_2-r_2)
|r_1-r_2|^{-\beta_0}dr_1dr_2\le C|s_1-s_2|^{-\beta_0}.
\end{equation}
Therefore,
\begin{align}
&\langle \Phi_{t,x}^{\varepsilon_1,\delta_1}, \Phi_{t,x}^{\varepsilon_2,\delta_2}\rangle_\mathcal H\le 
C\int_{[0,t]^2} \int_{\R^{2d}} p_{\varepsilon_1}(X_{s_1}^x-y_1)p_{\varepsilon_2}(X_{s_2}^x-y_2)
|s_1-s_2|^{-\beta_0}dy_1dy_2ds_1ds_2\notag\\
&=\frac{C}{(2\pi)^d}\int_{[0,t]^2} \int_{\R^d} \F\left(p_{\varepsilon_1}(X_{s_1}^x-\cdot)\right)(\xi) \overline{\F\left(p_{\varepsilon_2}(X_{s_2}^x-\cdot)\right)(\xi)}|s_1-s_2|^{-\beta_0}\mu(d\xi)ds_1ds_2 \notag\\
&=\frac{C}{(2\pi)^d} \int_{[0,t]^2} \int_{\R^d} \widehat p_{\varepsilon_1}(\xi)\widehat p_{\varepsilon_2}(\xi)\exp\big(-i\xi\cdot (X_{s_1}-X_{s_2})\big)|s_1-s_2|^{-\beta_0}\mu(d\xi)ds_1ds_2\notag\\
&\le C(\varepsilon_1, \varepsilon_2)\int_{[0,t]^2} |s_1-s_2|^{-\beta_0}ds_1ds_2<\infty.\label{e.3.1}
\end{align}
The second equality above holds because  $\F\left(\phi(\cdot-a)\right)(\xi)=\exp(-ia\cdot \xi)\widehat \phi(\xi)$. Note also that $C(\varepsilon_1, \varepsilon_2)= \int_{\R^d} \widehat p_{\varepsilon_1}(\xi)\widehat p_{\varepsilon_2}(\xi)\mu(d\xi)\le \int_{\R^d} \widehat p_{\varepsilon_1}(\xi) \|p_\varepsilon\|_1\mu(d\xi)=\int_{\R^d}p_{\varepsilon_1}(x)\gamma(x)dx<\infty$.  Hence, for $\varepsilon,\delta>0$, $\Phi_{t,x}^{\varepsilon,\delta}\in \mathcal H$ a.s. and  $W(\Phi_{t,x}^{\varepsilon,\delta})$ is well-defined a.s..

Now we show that $W(\Phi_{t,x}^{\varepsilon,\delta})$ forms a Cauchy sequence in $L^2$ when $(\varepsilon,\delta)\to 0$, for which it suffices to show that $\E[\langle \Phi_{t,x}^{\varepsilon_1,\delta_1}, \Phi_{t,x}^{\varepsilon_2,\delta_2}\rangle_\mathcal H]$ converges as $(\varepsilon_1, \delta_1)$ and $(\varepsilon_2, \delta_2)$ tend to zero. By the formula (\ref{inner'}) for the inner product using Fourier transforms,
\begin{align*}
\langle \Phi_{t,x}^{\varepsilon_1,\delta_1}, \Phi_{t,x}^{\varepsilon_2,\delta_2}\rangle_\mathcal H=&\frac{1}{(2\pi)^d}\int_{[0,t]^4} \int_{\R^d} \F\left(p_{\varepsilon_1}(X_{s_1}^x-\cdot)\right)(\xi) \overline{\F\left(p_{\varepsilon_2}(X_{s_2}^x-\cdot)\right)(\xi)}\\
&\qquad g_{\delta_1}(t-s_1-r_1)g_{\delta_2}(t-s_2-r_2)
|r_1-r_2|^{-\beta_0}\mu(d\xi)dr_1dr_2ds_1ds_2 \\
 =&\frac{1}{(2\pi)^d} \int_{[0,t]^4} \int_{\R^d}\widehat p_{\varepsilon_1}(\xi)\widehat p_{\varepsilon_2}(\xi)\exp\big(-i\xi\cdot (X_{s_1}-X_{s_2})\big)\\
&\qquad g_{\delta_1}(t-s_1-r_1)g_{\delta_2}(t-s_2-r_2)
|r_1-r_2|^{-\beta_0}\mu(d\xi)dr_1dr_2ds_1ds_2.
\end{align*}
By Fubini's theorem and thanks to (\ref{gdelta}) and Proposition \ref{integrable}, we can apply the dominated convergence theorem and get that
\begin{align}\E[\langle \Phi_{t,x}^{\varepsilon_1,\delta_1}, \Phi_{t,x}^{\varepsilon_2,\delta_2}\rangle_\mathcal H]\longrightarrow &\frac1{(2\pi)^d}\int_{[0,t]^2}\int_{\R^d} \E\exp\left(-i\xi\cdot(X_{s_1}-X_{s_2})\right)|s_1-s_2|^{-\beta_0}\mu(d\xi)ds_1ds_2\notag\\
& =\int_{[0,t]^2}|s_1-s_2|^{-\beta_0}\E\gamma(X_{s_2}-X_{s_1})ds_1ds_2\label{e3.4}
 \end{align}
as $(\varepsilon_1, \delta_1)$ and $(\varepsilon_2, \delta_2)$ go to zero.

Finally, conditional on $X$, $W(\Phi_{t,x}^{\varepsilon,\delta})$ is Gaussian and hence the limit (in probability) $W(\delta_0(X_{t-\cdot}^x-\cdot))$ is also Gaussian.  To show the formula (\ref{cc}) for conditional variance, it suffices to show that  
\begin{equation}
\langle \Phi_{t,x}^{\varepsilon,\delta}, \Phi_{t,x}^{\varepsilon,\delta}\rangle_\mathcal H\longrightarrow \int_{[0,t]^2}|s_1-s_2|^{-\beta_0}\gamma(X_{s_2}-X_{s_1})ds_1ds_2\label{e.3.4}
\end{equation} 
in $L^1(\Omega)$ as $(\varepsilon, \delta)\to 0$.  
 Noting that, by Lemma \ref{lemma4.2}, the inside integral in (\ref{e4.3})
\begin{align*}
&\int_{[0,t]^2} \int_{\R^{2d}} p_{\varepsilon}(X_{s_1}^x-y_1)p_{\varepsilon}(X_{s_2}^x-y_2)\gamma(y_1-y_2) \\ &~~~~~~~~~~~~~~g_{\delta}(t-s_1-r_1)g_{\delta}(t-s_2-r_2)
|r_1-r_2|^{-\beta_0}dy_1dy_2dr_1dr_2
\end{align*}
converges to $|s_1-s_2|^{-\beta_0} \gamma(X_{s_1}-X_{s_2})$ a.s. as $(\varepsilon, \delta)$ goes to zero, because of (\ref{e3.4}) we can apply Scheff\'e's lemma to get that the convergence is also in $L^1(\Omega\times [0,t]^2, P\times m)$ where $m$ is the Lebesgue measure on $[0,t]^2$. Consequently it follows that the convergence (\ref{e.3.4}) holds in $L^1(\Omega)$. 
\end{proof}
\begin{lemma}\label{lemma4.2} When $a-b\neq 0$,
$$\lim_{\varepsilon\to0}\int_{\R^{2d}} p_{\varepsilon}(a-y_1)p_{\varepsilon}(b-y_2)\gamma(y_1-y_2)dy_1dy_2=\gamma(a-b).$$
\end{lemma}
\begin{proof}
The change of variables $x_1=y_1-y_2, x_2=y_2$ implies that $\int_{\R^{2d}} p_{\varepsilon}(a-y_1)p_{\varepsilon}(b-y_2)\gamma(y_1-y_2)dy_1dy_2=\int_{\R^{2d}} p_{\varepsilon}(a-x_1-x_2)p_{\varepsilon}(b-x_2)\gamma(x_1)dx_1dx_2=\int_{\R^d} (p_{\varepsilon}*p_{\varepsilon})(a-b-x_1) \gamma(x_1)dx_1=\int_{\R^d}\frac1\varepsilon (p*p)(\frac{a-b-x}\varepsilon)\gamma(x)dx.$ Since the convolution $p*p$ is also a smooth probability density function with compact support, it suffices to prove the following result.
\end{proof}
\begin{lemma} Let $f_\varepsilon(x)=\frac{1}{\varepsilon^d}f(\frac{x}{\varepsilon})$, where $f\in C_0^\infty(\R^d)$ is a symmetric probability density function. Then we have
\[\lim_{\varepsilon\to0}\int_{\R^d} f_\varepsilon(a-x) \gamma(x)dx=\gamma(a), ~\forall a\neq 0.\]
\end{lemma}
\begin{proof}
Suppose that the support of the function $f$ is inside $[-M, M]$. Let the positive number $\varepsilon$ be sufficiently small such that $\gamma(x)$ is continuous on $[a-M\varepsilon, a+M\varepsilon]$. By the mean value theorem, we have 
\begin{align*}
&\int_{\R^d} f_\varepsilon(a-x)\gamma(x)dx=\int_{[a-M\varepsilon, a+M\varepsilon]}f_\varepsilon(a-x)\gamma(x)dx\\
&=\gamma(a_\varepsilon)\int_{[a-M\varepsilon, a+M\varepsilon]}f_\varepsilon(a-x)dx=\gamma(a_\varepsilon),
\end{align*}
where $a_\varepsilon\in [a-M\varepsilon, a+M\varepsilon]$. The result follows if we let $\varepsilon$ go to zero.
\end{proof}

\subsection{Feynman-Kac formula}
 For positive numbers $\varepsilon$ and $\delta$, define
\begin{align} \label{apprw}
\dot{W}^{\epsilon ,\delta
}(t,x):=\int_{0}^{t}\int_{\mathbb{R}^{d}}g_{\delta
}(t-s)p_{\epsilon }(x-y)W(ds,dy)=W(\phi_{t,x}^{\varepsilon, \delta}),
\end{align}
where $$\phi_{t,x}^{\varepsilon, \delta}(s,y)=g_{\delta
}(t-s)p_{\epsilon }(x-y)\cdot I_{[0,t]}(s).$$
Then $\dot{W}^{\epsilon ,\delta
}(t,x)$ exists in the classical sense and it is an approximation of $\dot W(t,x)$. Taking advantage of $\dot{W}^{\epsilon ,\delta
}(t,x)$, we can define the integral $\int_{0}^{T}\int_{\mathbb{R}^d}v(t,x)W(dt,dx)$ in the Stratonovich sense as follows.
\begin{definition}\label{defstra}
\label{def2} Suppose that  $v=\{v(t,x),t\geq 0,x\in
\mathbb{R}^d\}$ is a random field satisfying 
\begin{equation*}
\int_{0}^{T}\int_{\mathbb{R}^d}|v(t,x)|dxdt<\infty, \, \text{ a.s.},
\end{equation*}
and that the limit in probability $\lim\limits_{\epsilon ,\delta \downarrow 0}\int_{0}^{T}\int_{\mathbb{R}^d}v(t,x)
\dot{W}^{\epsilon ,\delta }(t,x)dxdt$  exists. The we denote the limit by 
\begin{equation*}
\int_{0}^{T}\int_{\mathbb{R}^d}v(t,x)W(dt,dx):=\lim_{\epsilon ,\delta \downarrow 0}\int_{0}^{T}\int_{\mathbb{R}^d}v(t,x)
\dot{W}^{\epsilon ,\delta }(t,x)dxdt.
\end{equation*}
and call it Stratonovich integral.
\end{definition}
Let $\mathcal F_t$ be the $\sigma$-algebra generated by $\{W(s,x), 0\le s\le t, x\in \R^d\}$, and we say that a random field $\{F(t,x),t\ge0,x\in \R^d\}$ is adapted  if $\{F(t,x), t\ge0\}$ is adapted to the filtration $\{\mathcal F_t\}_{t\ge 0}$ for all $x\in\R^d$. Denote the convolution between the function $q_t$ and $f$ by $Q_tf$, i.e., $$Q_tf(x):=\int_{\R^d} q_t(x-y)f(y)dy.$$ A mild solution to (\ref{spde}) in the Stratonovich sense is defined as follows. 

\begin{definition} \label{defmildstr} An adapted random field $u=\{u(t,x), t\ge0, x\in \R^d\}$ is a mild solution to (\ref{spde}) with initial condition $u_0\in C_b(\R^d)$, if for all $t\ge0$ and $x\in \R^d$ the following integral equation holds
\begin{equation}\label{mildstra}
u(t,x)=Q_tu_0(x)+\int_0^t\int_{\R^d}q_{t-s}(x-y)u(s,y)W(ds,dy),
\end{equation}
where the stochastic integral is in the Stratonovich sense of Definition \ref{defstra}.
\end{definition}

The following theorem is the main result in this section.
\begin{theorem}\label{thmfk}
Let the measure $\mu$ satisfy Hypothesis (I). Then 
\begin{equation}\label{mildu}
u(t,x)=\E^X\left[u_0(X_t^x)\exp\left(\int_0^t \int_{\R^d}\delta_0(X_{t-r}^x-y)W(dr, dy)\right)\right]
\end{equation}
is well-defined and it is a mild solution to (\ref{spde}) in the Stratonovich sense.
\end{theorem}
\begin{proof} 
Consider the following approximation of (\ref{spde})
\begin{equation}\label{appspde}
\begin{cases}
u^{\varepsilon, \delta}(t,x)=\L u^{\varepsilon, \delta}(t,x)+u^{\varepsilon, \delta}(t,x)\dot W^{\varepsilon, \delta}(t,x),\\
u^{\varepsilon,\delta}(0,x)=u_0(x).
\end{cases}
\end{equation}
By the classical Feynman-Kac formula, 
\[u^{\varepsilon, \delta}(t,x)=\E^X\left[u_0(X_t^x)\exp\left(\int_0^t \dot W^{\varepsilon, \delta}(r,X_{t-r}^x)dr \right)\right]=\E^X\left[u_0(X_t^x)\exp\left(W(\Phi_{t,x}^{\varepsilon,\delta})\right)\right]\] 
where $\Phi_{t,x}^{\varepsilon,\delta}$ is defined in (\ref{Phi}) and the last equality follows from the stochastic Fubini's theorem, is a mild solution to (\ref{appspde}), i.e.,
\begin{equation}\label{mildapp}
u^{\varepsilon, \delta}(t,x)=Q_tu_0(x)+\int_0^t\int_{\R^d}q_{t-s}(x-y)u^{\varepsilon, \delta}(s,y)\dot W^{\varepsilon, \delta}(s,y)dsdy.
\end{equation}

To prove the result, it suffices to show that as $(\varepsilon,\delta)$ tends to zero, both sides of  (\ref{mildapp}) converge   respectively in probability to those of (\ref{mildstra}) with $u(t,x)$ given in (\ref{mildu}). We split the proof in two steps for easier interpretation.

{\bf Step 1.} First, we show that $u^{\varepsilon, \delta}(t,x)\to u(t,x)$ in $L^p$ for all $p>1.$ By Theorem \ref{appfeyn}, as $(\varepsilon, \delta)\to 0$, $W(\Phi_{t,x}^{\varepsilon, \delta})$ converges to  $W(\delta_0(X_{t-\cdot}^x-\cdot)I_{[0,t]}(\cdot))$ in probability, and hence it suffices to show that  $$\sup\limits_{\varepsilon, \delta>0}\sup\limits_{t\in[0,T],x\in\R^d}\E[|u^{\varepsilon, \delta}(t,x)|^p]<\infty.$$ Note that $W(\Phi_{t,x}^{\varepsilon,\delta})$ is Gaussian conditional on $X$, and hence $$\E\left[\exp\left(pW(\Phi_{t,x}^{\varepsilon, \delta})\right)\right]=\E\left[\exp\left(\frac{p^2}{2}\|\Phi_{t,x}^{\varepsilon, \delta}\|_{\H}^2\right)\right].$$ By (\ref{inner'}) and (\ref{gdelta}), in a similar way of proving (\ref{e.3.1}), we can show that there exists a positive constant $C$ depending on $\beta_0$ only such that
\begin{align*}
 \|\Phi_{t,x}^{\varepsilon, \delta}\|_{\H}^{2}\le C\int_{[0,t]^2}\int_{\R^d}\left(\widehat p_\varepsilon(\xi)\right)^2\exp(-i\xi\cdot(X_r-X_s))|r-s|^{-\beta_0}\mu(d\xi)drds.
\end{align*} 
Therefore,
\begin{align*}
 \E[\|\Phi_{t,x}^{\varepsilon, \delta}\|_{\H}^{2n}]\le& C^n \int_{[0,t]^{2n}}\int_{\R^{nd}}\prod_{j=1}^n\left(\widehat p_\varepsilon(\xi_j)\right)^2\E\exp(-i\sum_{j=1}^n\xi_j\cdot(X_{r_j}-X_{s_j}))\\
 &\qquad \qquad\qquad  \prod_{j=1}^n|r_j-s_j|^{-\beta_0}\prod_{j=1}^n\mu(d\xi_j)dr_jds_j\\
 \le& C^n \int_{[0,t]^{2n}}\int_{\R^{nd}}\E\exp(-i\sum_{j=1}^n\xi_j\cdot(X_{r_j}-X_{s_j}))\prod_{j=1}^n|r_j-s_j|^{-\beta_0}\prod_{j=1}^n\mu(d\xi_j)dr_jds_j\\
 =&\E\left[\left(C \int_0^t \int_0^t |r-s|^{-\beta_0}\gamma(X_r-X_s)drds\right)^n\right].
\end{align*}
The second inequality above holds because $\sup_{\xi\in\R^d}\widehat p_\varepsilon(\xi)\le 1$ and  $\E\exp(-i\sum_{j=1}^n\xi_j\cdot(X_{r_j}-X_{s_j}))$ is a positive real number. Thus there is constant $C>0$ depending only on $\beta_0$ such that 
\[\sup_{\varepsilon,\delta>0}\sup\limits_{t\in[0,T],x\in\R^d}\E\left[\exp\left(\frac{p^2}{2}\|\Phi_{t,x}^{\varepsilon, \delta}\|_{\H}^2\right)\right]\le \E\left[\exp\left(C\frac{p^2}{2}\int_0^t \int_0^t |r-s|^{-\beta_0}\gamma(X_r-X_s)drds \right)\right],\]
where the term on the right-hand side is finite by Theorem \ref{expint}.

{\bf Step 2.} Now by Definition \ref{defstra}, it suffices to show that 
\[I^{\varepsilon, \delta}:=\int_0^t\int_{\R^d}q_{t-s}(x-y)(u^{\varepsilon, \delta}(s,y)-u(s,y))\dot W^{\varepsilon, \delta}(s,y)dsdy\]
converges in $L^2$ to zero. Denoting $v^{\varepsilon,\delta}_{s,y}=u^{\varepsilon, \delta}(s,y)-u(s,y)$ and noting that $\dot W^{\varepsilon, \delta}(s,y)=W(\phi_{s,y}^{\varepsilon,\delta})$ we have
\[\E[(I^{\varepsilon,\delta})^2]=\int_{[0,t]^2}\int_{\R^{2d}} q_{t-s_1}(x-y_1)q_{t-s_2}(x-y_2) \E\left[v^{\varepsilon,\delta}_{s_1, y_1}v^{\varepsilon,\delta}_{s_2, y_2}W(\phi_{s_1,y_1}^{\varepsilon, \delta})W(\phi_{s_2,y_2}^{\varepsilon, \delta})\right]dy_1dy_2ds_1ds_2.\]

Use the following notations $V_{t,x}^{\varepsilon,\delta}(X)=\int_0^t \dot W^{\varepsilon, \delta}(r,X_{t-r}^x)dr=W(\Phi_{t,x}^{\varepsilon,\delta}(X))$, $V_{t,x}(X)=\int_0^t \int_{\R^d}\delta_0(X_{t-r}^x-y)W(dr, dy)=W(\delta_0(X_{t-\cdot}^x-\cdot)I_{[0,t]}(\cdot))$, and \[A^{\varepsilon,\delta}(s_1, y_1, s_2, y_2)= \prod_{j=1}^2 u_0(X_{s_j}^{j}+y_j)\left[\exp\left(V_{s_j,y_j}^{\varepsilon,\delta}(X^j)\right)-\exp\left(V_{s_j,y_j}(X^j)\right)\right],\]
where $X^1$ and $X^2$ are two independent copies of $X$. Then 
\[\E\left[v^{\varepsilon,\delta}_{s_1, y_1}v^{\varepsilon,\delta}_{s_2, y_2}W(\phi_{s_1,y_1}^{\varepsilon, \delta})W(\phi_{s_2,y_2}^{\varepsilon, \delta})\right]=\E^{X^1, X^2}\left[\E^W\left[A^{\varepsilon,\delta}(s_1, y_1, s_2, y_2)W(\phi_{s_1,y_1}^{\varepsilon, \delta})W(\phi_{s_2,y_2}^{\varepsilon, \delta}) \right]\right].\]
By the integration by parts formula (\ref{intbypart}), 
\begin{align*}
&\E^W\left[A^{\varepsilon,\delta}(s_1, y_1, s_2, y_2)W(\phi_{s_1,y_1}^{\varepsilon, \delta})W(\phi_{s_2,y_2}^{\varepsilon, \delta})\right]\\
=&\E^W\left[\langle D^2 A^{\varepsilon,\delta}(s_1, y_1, s_2, y_2), \phi_{s_1,y_1}^{\varepsilon, \delta}\otimes \phi_{s_2,y_2}^{\varepsilon, \delta} \rangle_{\H^{\otimes2}}\right] + \E^W[A^{\varepsilon,\delta}(s_1, y_1, s_2, y_2)]\langle \phi_{s_1,y_1}^{\varepsilon, \delta}, \phi_{s_2,y_2}^{\varepsilon, \delta} \rangle_{\H},
\end{align*}
and hence we have 
\[\E\left[v^{\varepsilon,\delta}_{s_1, y_1}v^{\varepsilon,\delta}_{s_2, y_2}W(\phi_{s_1,y_1}^{\varepsilon, \delta})W(\phi_{s_2,y_2}^{\varepsilon, \delta})\right]= \E[A^{\varepsilon, \delta}(s_1, y_1, s_2, y_2) B^{\varepsilon, \delta}(s_1, y_1, s_2, y_2)]+\E[v^{\varepsilon,\delta}_{s_1, y_1}v^{\varepsilon,\delta}_{s_2, y_2}] \langle \phi_{s_1,y_1}^{\varepsilon, \delta}, \phi_{s_2,y_2}^{\varepsilon, \delta} \rangle_{\H},\]
where 
\begin{align*}
&B^{\varepsilon, \delta}(s_1, y_1, s_2, y_2)\\
=&\sum_{j,k=1}^2 \langle \phi_{s_1,y_1}^{\varepsilon, \delta},  \Phi_{s_j, y_j}^{\varepsilon,\delta}(X^j)-\delta(X^j_{s_j-\cdot}+y_j-\cdot )I_{[0,s_j]}(\cdot)\rangle_\H\langle \phi_{s_2,y_2}^{\varepsilon, \delta},  \Phi_{s_k, y_k}^{\varepsilon,\delta}(X^k)-\delta(X^k_{s_k-\cdot}+y_k-\cdot )I_{[0,s_k]}(\cdot)\rangle_\H.
\end{align*}
Therefore, 
\[\E[(I^{\varepsilon,\delta})^2]=J_1^{\varepsilon, \delta}+J_2^{\varepsilon, \delta},\]
with the notations
\begin{align*}
J_1^{\varepsilon, \delta}=&\int_{[0,t]^2}\int_{\R^{2d}} q_{t-s_1}(x-y_1)q_{t-s_2}(x-y_2) \E[A^{\varepsilon, \delta}(s_1, y_1, s_2, y_2) B^{\varepsilon, \delta}(s_1, y_1, s_2, y_2)]dy_1dy_2ds_1ds_2\\
\le& \int_{[0,t]^2}\int_{\R^{2d}} q_{t-s_1}(x-y_1)q_{t-s_2}(x-y_2) \\
&\qquad\qquad \left(\E[(A^{\varepsilon, \delta}(s_1, y_1, s_2, y_2))^2]\right)^{1/2} 
\left(\E[(B^{\varepsilon, \delta}(s_1, y_1, s_2, y_2))^2]\right)^{1/2}dy_1dy_2ds_1ds_2
\end{align*}
and 
\begin{align*}
J_2^{\varepsilon, \delta}=&\int_{[0,t]^2}\int_{\R^{2d}} q_{t-s_1}(x-y_1)q_{t-s_2}(x-y_2) \E[v^{\varepsilon,\delta}_{s_1, y_1}v^{\varepsilon,\delta}_{s_2, y_2}] \langle \phi_{s_1,y_1}^{\varepsilon, \delta}, \phi_{s_2,y_2}^{\varepsilon, \delta} \rangle_{\H}dy_1dy_2ds_1ds_2\\
\le & \int_{[0,t]^2}\int_{\R^{2d}} q_{t-s_1}(x-y_1)q_{t-s_2}(x-y_2) \left(\E[(v^{\varepsilon,\delta}_{s_1, y_1})^2]\right)^{1/2}\left(\E[(v^{\varepsilon,\delta}_{s_2, y_2})^2]\right)^{1/2} \langle \phi_{s_1,y_1}^{\varepsilon, \delta}, \phi_{s_2,y_2}^{\varepsilon, \delta} \rangle_{\H}dy_1dy_2ds_1ds_2.
\end{align*}

Now the problem is reduced to show that both $J_1^{\varepsilon,\delta}$ and $J_2^{\varepsilon,\delta}$ converge to zero as $(\varepsilon,\delta)\to 0.$ By the result in {\bf Step 1}, we have \[\lim_{\varepsilon,\delta\downarrow0}\E[(v_{s,y}^{\varepsilon,\delta})^2]=0,\] 
and similar arguments imply that 
\[\lim_{\varepsilon,\delta\downarrow0}\E[(A_{s,y}^{\varepsilon,\delta})^2]=0,\] 
for all $(s,y)\in[0,T]\times \R^d.$ Also note that both $\sup\limits_{\varepsilon,\delta>0}\sup\limits_{(s,y)\in[0,T]\times\R^d} \E[(v_{s,y}^{\varepsilon,\delta})^2]$ and  $\sup\limits_{\varepsilon,\delta>0}\sup\limits_{(s,y)\in[0,T]\times\R^d} \E[(A_{s,y}^{\varepsilon,\delta})^2]$ are finite. The fact that $\lim\limits_{\varepsilon,\delta\downarrow0}J_{1}^{\varepsilon, \delta}=0$  can be proven by by the dominated convergence theorem, noting that  Lemma \ref{lemma4.5} implies 
\[\sup_{\varepsilon,\delta>0}\sup_{(s_1, y_1)\in[0,T]\times\R^d}\sup_{(s_2, y_2)\in[0,T]\times\R^d}\E[(B^{\varepsilon, \delta}(s_1, y_1, s_2, y_2))^2]<\infty.\]

Now we show  $\lim\limits_{\varepsilon,\delta\downarrow 0}J_2^{\varepsilon, \delta}= 0$. By (\ref{inner'}) and (\ref{gdelta}), we have
\[\langle \phi_{s_1, y_1}^{\varepsilon, \delta}, \phi_{s_2, y_2}^{\varepsilon, \delta}\rangle_\mathcal H \le C |s_1-s_2|^{-\beta_0} \int_{\R^d} \exp\left(-i\xi\cdot(y_1-y_2)\right) \left(\widehat p_\varepsilon(\xi)\right)^2\mu(d\xi),\]
therefore,
\[J_2^{\varepsilon,\delta}\le C \int_0^t\int_0^t \int_{\R^{2d}}q_{t-s_1}(x-y_1)q_{t-s_2}(x-y_2)K^{\varepsilon,\delta}(s_1,y_1, s_2, y_2)|s_1-s_2|^{-\beta_0}dy_1dy_2ds_1ds_2\]
where
\begin{align*}
K^{\varepsilon,\delta}(s_1,y_1, s_2, y_2):=& \left(\E[(v^{\varepsilon,\delta}_{s_1, y_1})^2]\right)^{1/2}\left(\E[(v^{\varepsilon,\delta}_{s_2, y_2})^2]\right)^{1/2} \int_{\R^d}  \exp(-i\xi\cdot(y_1-y_2)) \left(\widehat p_\varepsilon(\xi)\right)^2\mu(d\xi) \\
\le&  C \int_{\R^d}  \exp(-i\xi\cdot(y_1-y_2)) \left(\widehat p_\varepsilon(\xi)\right)^2\mu(d\xi) .
\end{align*}
Denote
\[ L^{\varepsilon}_{s_1,s_2} :=\int_{\R^d}  \exp(-i\xi\cdot(y_1-y_2)) \left(\widehat p_\varepsilon(\xi)\right)^2\mu(d\xi).\]
 Hence 
 \begin{align*}
K^{\varepsilon,\delta}(s_1,y_1, s_2, y_2)\le C L^{\varepsilon}_{s_1,s_2} .
 \end{align*}
For the integral of  $L^{\varepsilon}_{s_1,s_2}$, we have
 \begin{align}
&\int_0^t\int_0^t \int_{\R^{2d}} q_{t-s_1}(x-y_1)q_{t-s_2}(x-y_2)L_{s_1,s_2}^{\varepsilon}|s_1-s_2|^{-\beta_0} dy_1dy_2ds_1ds_2\notag\\
&= \int_0^t\int_0^t \int_{\R^{2d}} \int_{\R^d} q_{t-s_1}(x-y_1)q_{t-s_2}(x-y_2)  \exp(-i\xi\cdot(y_1-y_2))\notag\\
&~~~~~~~~~~~~~~~~~~~~~~~~~~~~~\left(\widehat p_\varepsilon(\xi)\right)^2 |s_1-s_2|^{-\beta_0}\mu(d\xi)  dy_1dy_2ds_1ds_2\notag\\
&=  \int_0^t\int_0^t \int_{\R^{d}} \exp\left(-(t-s_1)\Psi(\xi)\right)\exp\left(-(t-s_2)\Psi(\xi)\right)\left(\widehat p_\varepsilon(\xi)\right)^2|s_1-s_2|^{-\beta_0}\mu(d\xi)ds_1ds_2 \notag\\
&\overset{\varepsilon\to 0}{\longrightarrow} \int_0^t\int_0^t \int_{\R^{d}} \exp\left(-(t-s_1)\Psi(\xi)\right)\exp\left(-(t-s_2)\Psi(\xi)\right)|s_1-s_2|^{-\beta_0}\mu(d\xi)ds_1ds_2 \notag\\
&=\int_0^t\int_0^t \int_{\R^{2d}} q_{t-s_1}(x-y_1)q_{t-s_2}(x-y_2) \gamma(y_1-y_2) |s_1-s_2|^{-\beta_0}dy_1dy_2ds_1ds_2w, \label{eq4.10}
 \end{align}
 where the convergence follows from the dominated convergence theorem, the last equality follows from the formula (\ref{exptransform'}), and  the last term is finite by Lemma \ref{lemma0}.  

We have shown that $K^{\varepsilon,\delta}(s_1,y_1, s_2, y_2)$ which converges to zero almost everywhere, is bounded by the sequence $L_{s_1,s_2}^\varepsilon$ which converges to $\gamma(y_1-y_2)$, and thanks to (\ref{eq4.10}), we can apply the generalized dominated convergence theorem to get that $\lim\limits_{\varepsilon,\delta\downarrow0}J_2^{\varepsilon,\delta}=0.$
\end{proof}
Using Theorem \ref{thmfk}, by direct computation we can get the following Feynman-Kac type  of representation for the moments of the solution to (\ref{spde}).
\begin{theorem}\label{thmfkmom}
Let $\mu$ satisfy Hypothesis (I), then the solution given by (\ref{mildu}) has finite moments of all orders. Furthermore, for any positive integer $p$,
 \begin{equation}\label{momstr}
 \E[u(t,x)^p]=\E\left[\prod_{j=1}^p u_0(X_t^j+x)\exp\left(\frac12\sum_{j,k=1}^p \int_0^t\int_0^t |r-s|^{-\beta_0}\gamma(X_r^j-X_s^k)drds\right)\right],
 \end{equation}
 where $X_1,\dots, X_p$ are $p$ independent copies of $X$.
\end{theorem}

In the proof of Theorem \ref{thmfk}, the following result is used in order to apply the dominated convergence theorem.

\begin{lemma}\label{lemma4.5} Let the measure $\mu$ satisfy  Hypothesis (I). Then, for any $n\in \mathbb N$,
 \[\sup_{\varepsilon,\delta>0}\sup_{\varepsilon',\delta'>0}\sup_{(s,y)\in[0,T]\times\R^d}\sup_{(r,z)\in[0,T]\times\R^d}\E\left[\left(\langle \phi_{s,y}^{\varepsilon, \delta},  \Phi_{r, z}^{\varepsilon',\delta'}(X)\rangle_\H\right)^n\right]<\infty,\]
 and 
 \[\sup_{\varepsilon,\delta>0}\sup_{(s,y)\in[0,T]\times\R^d}\sup_{(r,z)\in[0,T]\times\R^d}\E\left[\left(\langle \phi_{s,y}^{\varepsilon, \delta},  \delta(X_{r-\cdot}^z-\cdot)I_{[0,r]}(\cdot)\rangle_\H\right)^n\right]<\infty.\]
 \end{lemma}
\begin{proof}
First of all, $\langle \phi_{s,y}^{\varepsilon, \delta},  \Phi_{r, z}^{\varepsilon',\delta'}(X)\rangle_\H$ is a nonnegative real number by (\ref{inner}), and by (\ref{inner'})
\begin{align*}
 \langle \phi_{s,y}^{\varepsilon, \delta},  \Phi_{r, z}^{\varepsilon',\delta'}(X)\rangle_\H=&\int_0^r\int_0^s\int_0^r \int_{\R^d} \widehat p_{\varepsilon}(\xi)\widehat p_{\varepsilon'}(\xi)\exp(-i\xi\cdot(X_\tau^z-y)) g_{\delta'}(r-\mu-\tau)\\
 &\qquad \qquad \qquad g_{\delta}(s-\nu) |\mu-\nu|^{-\beta_0} \mu(\xi) d\tau d\mu d\nu.
\end{align*}
Therefore, denoting $D=[0,r]\times[0,s]\times[0,r]$, as in the first step of the proof for Theorem \ref{thmfk}, we have
 \begin{align*}
  &\E\left[\left(\langle \phi_{s,y}^{\varepsilon, \delta},  \Phi_{r, z}^{\varepsilon',\delta'}(X)\rangle_\H\right)^n\right]\\
  =& \int_{D^n}\int_{\R^{nd}}\prod_{j=1}^n  g_{\delta'}(r-\mu_j-\tau_j)g_\delta(s-\nu_j) |\mu_j-\nu_j|^{-\beta_0}\\
  &\qquad \qquad\prod_{j=1}^n \widehat p_{\varepsilon}(\xi_j)\widehat p_{\varepsilon'}(\xi_j)\E\left[\exp\left(-i\sum_{j=1}^n \xi_j\cdot (X_{\tau_j}^z-y)\right)\right]\mu(d\xi)d\tau d\mu d\nu\\
  \le & C^n \int_{[0,r]^n}\int_{\R^{nd}}\prod_{j=1}^n |r-s-\tau_j|^{-\beta_0} \exp\left(-\frac{\varepsilon+\varepsilon'}2\sum_{j=1}^n|\xi_j|^2\right)\left|\E\left[\exp\left(-i\sum_{j=1}^n \xi_j\cdot (X_{\tau_j}^z-y)\right)\right]\right|\mu(d\xi)d\tau \\
  \le & C^n \int_{[0,r]^n}\int_{\R^{nd}}\prod_{j=1}^n |r-s-\tau_j|^{-\beta_0} \E\left[\exp\left(-i\sum_{j=1}^n \xi_j\cdot X_{\tau_j}\right)\right]\mu(d\xi)d\tau.
 \end{align*}
 Thus, we have, denoting $\eta_j=\xi_j+\xi_{j+1}+\dots+\xi_n$,
 \begin{align*}
  &\E\left[\left(\langle \phi_{s,y}^{\varepsilon, \delta},  \Phi_{r, z}^{\varepsilon',\delta'}(X)\rangle_\H\right)^n\right]\\
  \le & C^n n! \int_{[0<\tau_1<\cdots<\tau_n<r]} \int_{\R^{nd}}\prod_{j=1}^n |r-s-\tau_j|^{-\beta_0} \E\left[\exp\left(-i\sum_{j=1}^n \xi_j\cdot X_{\tau_j}\right)\right]\mu(d\xi)d\tau\\
  = & C^n n! \int_{[0<\tau_1<\cdots<\tau_n<r]} \int_{\R^{nd}}\prod_{j=1}^n |r-s-\tau_j|^{-\beta_0} \E\left[\exp\left(-i\sum_{j=1}^n \eta_j\cdot (X_{\tau_j}-X_{\tau_{j-1}})\right)\right]\mu(d\xi)d\tau~(\text{let } \tau_0=0)\\
  = & C^n n! \int_{[0<\tau_1<\cdots<\tau_n<r]} \int_{\R^{nd}} \prod_{j=1}^n |r-s-\tau_j|^{-\beta_0} \exp\left(-\sum_{j=1}^n(\tau_j-\tau_{j-1})\Psi(\eta_j)\right)\mu(d\xi)d\tau\\
  \le &C^n n! \int_{[0<\tau_1<\cdots<\tau_n<r]} \int_{\R^{nd}}\prod_{j=1}^n |\tau_j+(s-r)|^{-\beta_0} \exp\left(-\sum_{j=1}^n(\tau_j-\tau_{j-1})\Psi(\xi_j)\right)\mu(d\xi)d\tau\text{ (by Lemma \ref{lemma4.6})}\\ 
  =:&C^n n! U_n(r,s) 
       \end{align*}
When $s-r\ge 0$, for all $0<r\le s <T,$
\begin{align*}
U_n(r,s) \le&  \int_{[0<\tau_1<\cdots<\tau_n<r]} \int_{\R^{nd}}\prod_{j=1}^n |\tau_j|^{-\beta_0} \exp\left(-\sum_{j=1}^n(\tau_j-\tau_{j-1})\Psi(\xi_j)\right)\mu(d\xi)d\tau\\
\le &\int_{[0<\tau_1<\cdots<\tau_n<T]} \int_{\R^{nd}}\prod_{j=1}^n |\tau_j-\tau_{j-1}|^{-\beta_0} \exp\left(-\sum_{j=1}^n(\tau_j-\tau_{j-1})\Psi(\xi_j)\right)\mu(d\xi)d\tau.
\end{align*}
By Proposition \ref{hhnt}, $U_n(r,s)$ is uniformly bounded by a finite number depending on $(T, n, \beta_0)$ and the measure $\mu$ only.

When $r-s>0$, the set $[0<\tau_1<\dots<\tau_n<r]$ is the union of $A_k's$ for  $k=0,1,2,\dots, n$ where $A_k=[0=\tau_0<\tau_1<\dots<\tau_k<r-s<\tau_{k+1}<\dots< \tau_n<r]$. On each $A_k$, we have 
\begin{align*}
&\int_{A_k} \int_{\R^{nd}}\prod_{j=1}^n |r-s-\tau_j|^{-\beta_0} \exp\left(-\sum_{j=1}^n(\tau_j-\tau_{j-1})\Psi(\xi_j)\right)\mu(d\xi)d\tau\\
=&\int_{A_k} \int_{\R^{nd}}\prod_{j=1}^k (r-s-\tau_j)^{-\beta_0} \exp\left(-\sum_{j=1}^k(\tau_j-\tau_{j-1})\Psi(\xi_j)\right)\\
&\qquad \qquad (\tau_{k+1}-(r-s))^{-\beta_0} \exp\left(-(\tau_{k+1}-(r-s)+(r-s)-\tau_{k})\Psi(\xi_j)\right)\\
&\qquad \qquad  \prod_{j=k+2}^n (\tau_j-(r-s))^{-\beta_0} \exp\left(-\sum_{j=k+2}^n(\tau_j-\tau_{j-1})\Psi(\xi_j)\right)\mu(d\xi)d\tau\\
\le&\int_{A_k} \int_{\R^{nd}}\prod_{j=1}^k (r-s-\tau_j)^{-\beta_0} \exp\left(-\sum_{j=1}^k(\tau_j-\tau_{j-1})\Psi(\xi_j)\right)\\
&\qquad \qquad (\tau_{k+1}-(r-s))^{-\beta_0} \exp\left(-(\tau_{k+1}-(r-s))\Psi(\xi_{k+1})\right)\\
&\qquad \qquad  \prod_{j=k+2}^n (\tau_j-(r-s))^{-\beta_0} \exp\left(-\sum_{j=k+2}^n(\tau_j-\tau_{j-1})\Psi(\xi_j)\right)\mu(d\xi)d\tau\\
=&\int_{[0<\tau_1<\dots<\tau_k<r-s]}\int_{\R^{kd}} \prod_{j=1}^k (r-s-\tau_j)^{-\beta_0} \exp\left(-\sum_{j=1}^k(\tau_j-\tau_{j-1})\Psi(\xi_j)\right)\mu(d\xi) d\tau\\
&\times \int_{[r-s<\tau_{k+1}<\dots<\tau_n<r]}\int_{\R^{(n-k)d}} \prod_{j=k+1}^n (\tau_j-(r-s))^{-\beta_0} \exp\left(-(\tau_{k+1}-(r-s))\Psi(\xi_{k+1})\right)\\
&\qquad \qquad~~~~~~~~~~~~~~~~ \exp\left(-\sum_{j=k+2}^n(\tau_j-\tau_{j-1})\Psi(\xi_j)\right)\mu(d\xi)d\tau\\
=:& M_1(s,r)\times M_2(s,r).
\end{align*}
By Lemma \ref{lemma1}, we have 
\begin{align*}
\sup_{0<s<r<T}M_1(s,r)&\le C^k \sup_{0<s<r<T}\int_{[0<\tau_1<\dots<\tau_k<r-s]} \prod_{j=1}^k (r-s-\tau_j)^{-\beta_0} (1+(\tau_j-\tau_{j-1})^{-1+\beta_0}) d\tau\\
&<\infty.
\end{align*}
For $M_2(s,r)$, let $\theta_{j}=\tau_{j}-(r-s), j=k+1,\dots, n$, and assume $\theta_k=0$, then   for all $0<s<r<T$, 
   \begin{align*}
   M_2(s,r)&=\int_{[0<\theta_{k+1}<\dots<\theta_n<s]}\int_{\R^{(n-k)d}}\prod_{j=k+1}^n  \theta_j^{-\beta_0} \exp\left(-\sum_{j=k+1}^n(\theta_j-\theta_{j-1})\Psi(\xi_j)\right)\mu(d\xi)d\theta\\
   &\le  \int_{[0<\theta_{k+1}<\dots<\theta_n<T]}\int_{\R^{(n-k)d}}\prod_{j=k+1}^n ( \theta_j-
   \theta_{j-1})^{-\beta_0} \exp\left(-\sum_{j=k+1}^n(\theta_j-\theta_{j-1})\Psi(\xi_j)\right)\mu(d\xi)d\theta,
   \end{align*}
  and the last integral is bounded by a finite number depending on $(n-k, T, \beta_0)$ and $\mu$ by Proposition \ref{hhnt}. 
   
Thus we have shown that when $r-s>0, \sup_{0<s<r<T} U_n(r,s)<\infty$, and the first inequality is obtained. Finally, since  $\langle \phi_{s,y}^{\varepsilon, \delta},  \Phi_{r, z}^{\varepsilon',\delta'}(X)\rangle_\H$ converges to $\langle \phi_{s,y}^{\varepsilon, \delta},  \delta(X_{r-\cdot}^z-\cdot)I_{[0,r]}(\cdot)\rangle_\H$ in probability as $(\varepsilon',\delta')\to 0$, the second the inequality follows from the first one and Fatou's lemma. \hfill
\end{proof}

In the proofs of the previous lemma, Theorem \ref{thmfkskr} and Theorem \ref{thmholder'}, the following maximum principle is needed.

\begin{lemma}\label{lemma4.6}
For any $t>0$ and  $a\in \R^d,$
\[\int_{\R^d} \exp(-t \Psi(\xi+a))\mu(d\xi)\le \int_{\R^d}\exp(-t\Psi(\xi))\mu(d\xi).
   \]
   \end{lemma}
   
   \begin{proof} The result follows directly from Lemma \ref{lemma}.
\end{proof}

\subsection{H\"older continuity}
\begin{hypothesis}[S1]
The spectral measure $\mu$ satisfies that for all $z\in \R^d$, there exist $\alpha_1\in(0,1]$ and $C>0$ such that
\[\int_0^T\int_0^T \int_{\R^d}|r-s|^{-\beta_0}e^{-|r-s|\Psi(\xi)}\left(1-e^{-i\xi\cdot z}\right)\mu(d\xi)drds\le C|z|^{2\alpha_1}.\]
\end{hypothesis}
\begin{hypothesis}[T1]  The spectral measure $\mu$ satisfies that for all $a$ in a bounded subset of $\R$, there exist $\alpha_2\in(0,1]$ and $C>0$ such that 
\[\int_0^T\int_0^T \int_{\R^d}|r-s|^{-\beta_0}\bigg|\exp\Big(-|r-s|\Psi(\xi)\Big)-\exp\Big(-|r-s+a|\Psi(\xi)\Big)\bigg|\mu(d\xi)drds\le C |a|^{\alpha_2}.\]
\end{hypothesis}
\begin{remark}
A sufficient condition for  Hypothesis (S1) to hold is the following
\begin{equation}\label{s1'}
\int_{\R^d}\frac{|\xi|^{2\alpha_1}}{1+(\Psi(\xi))^{1-\beta_0}}\mu(d\xi)<\infty
\end{equation}
due to Lemma \ref{lemma0} and the fact that $1-\cos x\le |x|^{2\alpha_1}$. Note that $\alpha_1<1-\beta_0$ is a necessary condition for (\ref{s1'}) to hold. This is because $\mu(A)<\infty$ for any bounded set $A\subset \R^d$, $\mu(\R^d)=\gamma(0)=\infty$, $\lim_{\xi\to\infty} \Psi(\xi)=\infty$ and $\limsup_{\|\xi\|\to\infty}\frac{\Psi(\xi)}{\|\xi\|^2}<\infty.$

Similarly, a sufficient condition for  Hypothesis (T1) to be true is that 
\begin{equation}\label{t1'}
\int_{\R^d}\frac{(\Psi(\xi))^{\alpha_2}}{1+(\Psi(\xi))^{1-\beta_0}}\mu(d\xi)<\infty
\end{equation}
because of Remark \ref{rm3.3} and the fact that  $|e^{-x}-e^{-y}|\le (e^{-x}+e^{-y})|x-y|^{\alpha}$ for $x,y\ge 0$ and $\alpha\in (0,1]$. Indeed for $a>0, e^a-1\le (e^a+1)(a\wedge 1)$, and hence  $e^a-1\le (e^a+1)a^\alpha$ for $\alpha\in(0,1]$.  One necessary condition for (\ref{t1'}) to hold is $\alpha_2<1-\beta_0$.
\end{remark}
\begin{theorem} \label{thmholder} Let $u_0(x)\equiv1$.  If  the measure $\mu$ satisfies Hypothesis (S1), then the solution $u(t,x)$ given by the Feynman-Kac formula (\ref{mildu}) has a version that is $\theta_1$-H\"older continuous in $x$ on any compact set of $[0,\infty)\times \R^d$, with $\theta_1<\alpha_1$;  Similarly, if $\mu$ satisfies Hypothesis (T1), the solution $u(t,x)$ has a version that is $\theta_2$-H\"older continuous in $t$ on any compact set of $[0,\infty)\times \R^d$, with $\theta_2<\alpha_2/2$.
\end{theorem}

\begin{remark} \label{RemarkHolder}The above theorem coincides with  Theorem 4.9 in \cite{HHNT} when $\L=\frac12\Delta$ and $\dot W$ is a general Gaussian noise. Now we consider the case when $\L=-(-\Delta)^{\alpha/2}$ with  $\alpha\in (0,2]$, i.e., $\psi(\xi)=|\xi|^\alpha$,  and  $\gamma(x)=|x|^{-\beta}, \beta\in(0,d)$ or $\gamma(x)=\prod_{j=1}^d |x_j|^{-\beta_j}, \beta_j\in(0,1), j=1,\dots, d$. Note that the Fourier transform $\widehat \gamma(\xi)$ is $|\xi|^{\beta-d}$ or $\prod_{j=1}^d |\xi_j|^{\beta_j-1}$ up to a multiplicative constant. We also denote $\beta=\sum_{j=1}^d \beta_j$.    Then condition \eqref{s1'} is equivalent to $\alpha_1<\frac12\left[\alpha(1-\beta_0)-\beta\right]$ and condition \eqref{t1'} is equivalent to $\alpha_2<(1-\beta_0)-\frac\beta\alpha.$ Assuming conditions \eqref{s1'} and \eqref{t1'}, Theorem \ref{thmholder} coincides with the results obtained in \cite{MR2778803} and \cite{CHS}, i.e., on any compact set of  $[0,\infty)\times \R^d$, the solution $u(t,x)$ has a version that is $\theta_1$-H\"older continuous in $x$ with $\theta_1\in (0, \frac12\left[\alpha(1-\beta_0)-\beta\right]) $ and $\theta_2$-H\"older continuous in $t$ with $\theta_2\in(0, \frac12\left[(1-\beta_0)-\frac\beta\alpha\right].$
\end{remark}
\begin{proof}
Recall that 
$V_{t,x}= \int_0^t \int_{\mathbb R^d} \delta(X_{t-s}^x-y)W(ds,dy).$
Noting that $|e^a-e^b|\le (e^a+e^b)|a-b|$, we have for any $p>0$
\begin{align*}
&\E^W[\left|\E^X\left[\exp(V_{t,x})-\exp(V_{s,y})\right]\right|^p]\le C\E^W\left[\left(\E^X\left[\exp(2V_{t,x})+\exp(2V_{s,y})\right]\right)^{p/2}\left(\E^X\left[|V_{t,x}-V_{s,y}|^2\right]\right)^{p/2}\right] \\
&\le C \E\left[\exp(pV_{t,x})+\exp(pV_{s,y})\right] \left(\E^W\left[\left(\E^X\left[|V_{t,x}-V_{s,y}|^2\right]\right)^{p}\right]\right)^{1/2}.
\end{align*}
By Theorem \ref{expint}, $\E\left[\exp(pV_{t,x})+\exp(pV_{t,y})\right]<\infty$. On the other hand,
$$\left(\E^W\left[\left(\E^X\left[|V_{t,x}-V_{s,y}|^2\right]\right)^{p}\right]\right)^{1/2}\le \left(\E^X \left(\E^W[|V_{t,x}-V_{s,y}|^{2p}]\right)^{1/p}\right)^{p/2}\le C_p \left(\E^X \E^W[|V_{t,x}-V_{s,y}|^{2}]\right)^{p/2}, $$
where the first inequality follows from Minkowski's inequality and the second one holds because of the equivalence between the  $L^p$-norm and $L^2$-norm of Gaussian random variables.  For the spatial estimate, by Hypothesis (S1), 
\begin{align*}
&\E^X \E^W[|V_{t,x}-V_{t,y}|^{2}]=2\int_0^t \int_0^t |r-s|^{-\beta_0}\E^X\left[\gamma(X_r-X_s)-\gamma(X_r-X_s+x-y))\right]drds\\
&=2\int_0^t\int_0^t \int_{\R^d} |r-s|^{-\beta_0}e^{-|r-s|\Psi(\xi)}\left(1-e^{-i\xi\cdot (x-y)}\right)\mu(d\xi) drds\le C|x-y|^{2\alpha_1}.
\end{align*}
Therefore \[\E^W[\left|\E^X\left[\exp(V_{t,x})-\exp(V_{t,y})\right]\right|^p]\le C_p |x-y|^{\alpha_1 p},\]
and the H\"older continuity of $u(t,x)$ in space follows from Komogorov's continuity criterion.

Now assume that $0\le s<t\le T,$ then 
\begin{align*}
&\E[(V_{t,x}-V_{s,x})^2] \\
=&\E\left[\left(\int_0^s\int_{\R^d} \left(\delta_0(X_{t-r}^x-z)-\delta_0(X_{s-r}^x-z)\right)W(dr, dz)+\int_s^t\int_{\R^d}\delta_0(X_{t-r}^x-z)W(dr,dx)\right)^2\right] \\
\le & 2(A+B),
\end{align*}
where \[A=\E\left[\left(\int_0^s\int_{\R^d} \left(\delta_0(X_{t-r}^x-z)-\delta_0(X_{s-r}^x-z)\right)W(dr, dz)\right)^2 \right],\]
and \[B=\E\left[\left(\int_s^t\int_{\R^d}\delta_0(X_{t-r}^x-z)W(dr,dx)\right)^2\right]. \]
For the first term $A$, by Hypothesis (T1), we have 
\begin{align*}
A&=\E\left[\int_0^s\int_0^s |s_1-s_2|^{-\beta_0}\big[\gamma(X_{t-s_1}-X_{t-s_2})+\gamma(X_{s-s_1}-X_{s-s_2})-2\gamma(X_{t-s_1}-X_{s-s_2})\big]ds_1ds_2\right]\\
&\le 2 \int_0^s\int_0^s \int_{\R6d}|s_1-s_2|^{-\beta_0}\bigg|\exp\Big(-|s_1-s_2|\Psi(\xi)\Big)-\exp\Big(-|t-s-s_1+s_2|\Psi(\xi)\Big)\bigg|\mu(d\xi)ds_1ds_2\\
&\le C|t-s|^{\alpha_2}.
\end{align*}
For the term $B$, we have
\begin{align*}
 &B=\int_s^t\int_s^t |s_1-s_2|^{-\beta_0}\E\gamma(X_{s_1}-X_{s_1})ds_1ds_2 =\int_0^{t-s}\int_0^{t-s} |s_1-s_2|^{-\beta_0}\E\gamma(X_{s_1}-X_{s_1})ds_1ds_2\\
 &=\int_{\R^d}\int_0^{t-s}\int_0^{t-s} |s_1-s_2|^{-\beta_0}\exp\left(-|s_1-s_2|\Psi(\xi)\right)ds_1ds_2\mu(d\xi).
\end{align*}
By Lemma \ref{lemma0}, we have that for $(t-s)$ in a bounded domain, there exists a constant $C$ such that 
\[\int_0^{t-s}\int_0^{t-s} |s_1-s_2|^{-\beta_0}\exp\left(-|s_1-s_2|\Psi(\xi)\right)ds_1ds_2\le C(t-s) \frac{1}{1+(\Psi(\xi))^{1-\beta_0}}.\]
Hence $B\le C (t-s)$, and 
\[\E^W[\left|\E^X\left[\exp(V_{t,x})-\exp(V_{s,x})\right]\right|^p]\le C\left(\E[(V_{t,x}-V_{s,x})^2]\right)^{p/2}\le C(A+B)^{p/2}\le C (t-s)^{p\alpha_2/2}.\]
The H\"older continuity in time is obtained by Kolmogorov's criterion. \hfill
\end{proof}

\section{Skorohod equation}\label{sectionsk}
In this section, we consider (\ref{spde}) in the Skorohod sense, i.e., we consider the following SPDE,
\begin{equation}\label{spde'}
\begin{cases}
\displaystyle\frac{\partial u}{\partial t}=\L u+u\diamond\dot {W},& t\ge 0, x\in \mathbb R^d\\
u(0,x)=u_0(x),& x\in\mathbb R^d,
\end{cases}
\end{equation}
where the symbol $\diamond$ means the Wick product.

\subsection{Existence and uniqueness of the mild solution}
In this subsection, we will obtain the existence and uniqueness of the mild solution to (\ref{spde'}) under the following assumption. 
\begin{hypothesis}[II] \label{h2} The spectral measure $\mu$ satisfies
\[\int_{\R^d}\frac1{1+\Psi(\xi)}\mu(d\xi)<\infty.\]
\end{hypothesis}
 
\begin{remark}
When $\L=-(-\Delta)^{\alpha/2}$ for $\alpha\in (0,2]$ and $\gamma(x)$ is  $\prod_{j=1}^d|x_j|^{-\beta_j}$ or $|x|^{-\beta}$, Hypothesis (II) is equivalent to $\beta<\alpha$. It is also a necessary condition for (\ref{spde'}) to have a unique mild solution (\cite{MR3262944}).
\end{remark}

\begin{definition} An adapted random field $u=\{u(t,x), t\ge0, x\in \R^d\}$ is a mild solution to (\ref{spde'}) with initial condition $u_0\in C_b(\R^d)$, if for all $t\ge0$ and $x\in \R^d$, $\E[u^2(t,x)]<\infty$, and the following integral equation holds
\begin{equation}\label{mildsk}
u(t,x)=Q_tu_0(x)+\int_0^t\int_{\R^d}q_{t-s}(x-y)u(s,y) W^\diamond(ds,dy),
\end{equation}
\end{definition}
\noindent where the stochastic integral is in the Skorohod sense.

Suppose that $u=\{u(t,x), t\ge0, x\in \R^d\}$ is a solution to (\ref{mildsk}), then for fixed $(t,x)$, the square integrable random variable $u(t,x)$ can be expressed uniquely as the Wiener chaos expansion,
\begin{equation}\label{chaos}
 u(t,x)=\sum_{n=0}^\infty I_n(f_n(\cdot, t,x)),
\end{equation} 
where $f_n(\cdot, t,x)$ is symmetric in $\mathcal H^{\otimes n}.$
On the other hand, if we apply (\ref{mildsk}) repeatedly, as in \cite{MR2449130, MR2473265}, we can find explicit representations for $f_n$ with $n\ge 1$
\[f_n(s_1, x_1,\dots, s_n, x_n, t,x)=\frac{1}{n!} q_{t-s_{\sigma(n)}}(x-x_{\sigma(n)})\cdots q_{s_{\sigma(2)}-s_{\sigma(1)}}(x_{\sigma(2)}-x_{\sigma(1)}) Q_{s_{\sigma(1)}}u_0(x_{\sigma(1)}).\]
Here $\sigma$ denotes the permutation of $\{1,2,\dots, n\}$ such that $0<s_{\sigma(1)}<\dots<s_{\sigma(n)}<t$. Note that $f_0(t,x)=Q_tu_0(x)$.

Therefore, to obtain the existence and uniqueness of the solution to (\ref{mildsk}), it suffices to prove 
\begin{equation}\label{sumchaos}
 \sum_{n=0}^\infty n! \|f_n(\cdot, t,x)\|^2_{\H^{\otimes n}}<\infty, ~~ \forall\, (t,x)\in[0,T]\times \R^d.
\end{equation}

\begin{theorem}\label{thmfkskr}
Let  the measure $\mu$ satisfy  Hypothesis (II).
Then (\ref{sumchaos}) holds, and consequently, $u(t,x)$ given by (\ref{chaos}) is the   unique mild solution to (\ref{spde'}).
\end{theorem}

\begin{proof}  As $u_0$ is a bounded and continuous function, without loss of generality, we assume that $u_0(x)\equiv1$. Now we have
 \begin{align*}
  &n! \|f_n(\cdot, t,x)\|^2_{\H^{\otimes n}}\\
  = &  n! \int_{[0,t]^{2n}}\int_{\R^{2nd}} h_n(s,y,t,x)h_n(r,z, t,x)\prod_{j=1}^n |s_j-r_j|^{-\beta_0}\prod_{j=1}^n \gamma(y_j-z_j) dydzdsdr,
 \end{align*}
where 
\begin{equation}\label{hn}
h_n(s_1,\dots, s_n, x_1,\dots, x_n, t,x)=\frac{1}{n!} q_{t-s_{\sigma(n)}}(x-x_{\sigma(n)})\cdots q_{s_{\sigma(2)}-s_{\sigma(1)}}(x_{\sigma(2)}-x_{\sigma(1)}).
\end{equation}
Then by (\ref{inner}) and (\ref{exptransform'}), 
\begin{align*}
  &n! \|f_n(\cdot, t,x)\|^2_{\H^{\otimes n}}\\
  \le &  n! \int_{[0,t]^{2n}}\int_{\R^{nd}} \F h_n(s,
  \cdot,t,x)(\xi)\overline{\F h_n(r,\cdot, t,x)(\xi)}  \mu(d\xi) \prod_{j=1}^n|s_j-r_j|^{-\beta_0}dsdr,\\
  \le &  n! \int_{[0,t]^{2n}} A_{t,x}(s) A_{t,x}(r) \prod_{j=1}^n|s_j-r_j|^{-\beta_0}dsdr\\
  \le &  n! \int_{[0,t]^{2n}} A^2_{t,x}(s) \prod_{j=1}^n|s_j-r_j|^{-\beta_0}dsdr, ~\text{(using $2ab\le a^2+b^2$ and the symmetry of the integral)}
 \end{align*}
where 
\begin{equation}\label{atx}
A_{t,x}(s)=\left(\int_{\R^{nd}} |\F h_n(s,
  \cdot,t,x)(\xi)|^2 \mu(d\xi)\right)^{1/2},
  \end{equation}
  with 
  \begin{equation}
   \F h_n(s,  \cdot,t,x)(\xi)=\frac1{n!} e^{-ix\cdot (\xi_1+\dots+\xi_n)}\prod_{j=1}^n \exp\Big[-[s_{\sigma(j+1)}-s_{\sigma(j)}]\Psi(\xi_{\sigma(1)}+\dots+\xi_{\sigma(j)})\Big],
  \end{equation}
where we use the convention $s_{\sigma(n+1)}=t$.

Note that $\int_0^t\int_0^t f(s)|s-r|^{-\beta_0}dsdr\le 2\int_0^t r^{-\beta_0}dr \int_0^t |f(s)|ds$ and let 
$D_t= 2\int_0^t r^{-\beta_0}dr$. Therefore,
\begin{align*}
 &n! \|f_n(\cdot, t,x)\|^2_{\H^{\otimes n}}\le  D_t^n n! \int_{[0,t]^{n}} A^2_{t,x}(s) ds\\
 =&  D_t^nn!\int_{[0,t]^{n}}\int_{\R^{nd}} |\F h_n(s,\cdot, t,x)(\xi)|^2\mu(d\xi) ds\\
 =&  D_t^n\frac1{n!}\int_{[0,t]^{n}}\int_{\R^{nd}} \prod_{j=1}^n \exp\Big[-2[s_{\sigma(j+1)}-s_{\sigma(j)}]\Psi(\xi_{\sigma(1)}+\cdots+ \xi_{\sigma(j)})\Big]\mu(d\xi) ds\\
 \le &  D_t^n \frac{1}{n!}\int_{[0,t]^{n}}\int_{\R^{nd}}\prod_{j=1}^n  \exp\Big[-2[s_{\sigma(j+1)}-s_{\sigma(j)}]\Psi( \xi_{\sigma(j)})\Big]\mu(d\xi) ds ~ \text{ (by Lemma \ref{lemma4.6}) }\\
 = & D_t^n\int_{[0<s_1<\dots<s_n<t]}\int_{\R^{nd}}\prod_{j=1}^n  \exp\Big[-2[s_{j+1}-s_{j}]\Psi( \xi_{j})\Big]\mu(d\xi) ds.
   \end{align*}
Similar as in the proof of Theorem \ref{expint}, we can apply Proposition \ref{hhnt} with $\beta_0=0$ for the last integral and then get the following estimate 
$$n!\|f_n(\cdot, t,x)\|^2_{\H^{\otimes n}}\le D_t^n\sum_{k=0}^n \binom nk \frac{t^k}{k!}m_N^{k}\left[A_0\varepsilon_N\right]^{n-k}, $$
where $\varepsilon_N$ and $m_N$ are given in (\ref{e3.5}) with $\beta_0=0$.  Hence, if we choose $N$ sufficiently large such that $2D_t A_0\varepsilon_N <1$, then we have  
\begin{align*}
 &\sum_{n=0}^\infty n!\|f_n(\cdot, t,x)\|^2_{\H^{\otimes n}}\le  \sum_{n=0}^\infty  D_t^n\sum_{k=0}^n \binom nk \frac{t^k}{k!}m_N^{k}\left[A_0\varepsilon_N\right]^{n-k}\\
 \le & \sum_{k=0}^\infty \frac{t^k}{k!}m_N^{k} \sum_{n=k}^\infty D_t^n 2^n \left[A_0\varepsilon_N\right]^{n-k}=  \frac{1}{1-2D_t A_0\varepsilon_N}\sum_{k=0}^\infty \frac{t^k}{k!}m_N^{k}D_t^k 2^k<\infty.\qedhere
\end{align*}
\end{proof}
\begin{remark}
 Let $\eta(x)$ be a locally integrable function, then as in \cite{HHNT}, the result of the above theorem still holds if the temporal kernel $|r-s|^{-\beta_0}$ is replaced by  $\eta(r-s)$.
\end{remark}

The following theorem provides the Feynman-Kac type of representations for the solution and the moments of the solution when the spectral measure $\mu$ satisfies the stronger condition Hypothesis (I). The proof is similar to the one in \cite{MR2778803} and we omit it here.
\begin{theorem}\label{thmfkmom'}
If we assume that $\mu$ satisfy  Hypothesis (I), then 
 \begin{equation}\label{fkskr}
  u(t,x)=\E^X\left[u_0(X_t^x) \exp\left(\int_0^t \int_{\R^d} \delta_0(X_{t-r}^x-y)W(dr,dy)-\frac12 \int_0^t \int_0^t |r-s|^{-\beta_0}\gamma(X_r-X_s)drds \right)\right]
  \end{equation}
is the unique mild solution to (\ref{spde'}) in the Skorohod sense.
 Consequently, for any positive integer $p$, we have
 \begin{equation}\label{momskr}
  \E[u(t,x)^p]=\E\left[\prod_{j=1}^p u_0(X_t^j+x)\exp\left(\sum_{1\le j<k\le p} \int_0^t\int_0^t |r-s|^{-\beta_0}\gamma(X_r^j-X_s^k)drds\right)\right],
  \end{equation}
where $X_1,\dots, X_p$ are $p$ independent copies of $X$.
\end{theorem}
\subsection{Feynman-Kac formula for the moments of the solution}
When the measure $\mu$ satisfies Hypothesis (II)  but not  Hypothesis (I), the representation (\ref{fkskr}) may be invalid since $\int_0^t \int_0^t |r-s|^{-\beta_0}\gamma(X_r-X_s)drds$ might be infinite a.s. (see \cite{MR2778803} for the case that $X$ is a $d$-dimensional Brownian motion and $\gamma(x)=\prod_{j=1}^d|x_j|^{-\beta_j}, \beta_j\in (0,1), j=1, \dots, d$). However, the Feynman-Kac formula (\ref{momskr}) for the moments still holds as stated in the following theorem.
\begin{theorem}
 Let the measure $\mu$ satisfy  Hypothesis (II), then the Feynman-Kac formula (\ref{momskr}) for the moments of the mild solution  to (\ref{spde'}) holds.
\end{theorem}

\begin{proof}
We will adopt the approximation method used in \cite[Section 5]{MR2449130} to prove the result. The proof is split into three steps for easier reading.

{\bf Step 1.} Consider the approximation of (\ref{spde'}),
\begin{equation}\label{appspde'}
\begin{cases}
u^{\varepsilon, \delta}(t,x)=\L u^{\varepsilon, \delta}(t,x)+u^{\varepsilon, \delta}(t,x)\diamond\dot W^{\varepsilon, \delta}(t,x),\\
u^{\varepsilon,\delta}(0,x)=u_0(x).
\end{cases}
\end{equation}
Recall that $\dot W^{\varepsilon, \delta}(t,x)$ is defined in (\ref{apprw}). If $u^{\varepsilon, \delta}(t,x)\in\mathbb D^{1,2}$, then by (\ref{wick})
\begin{equation*}
 u^{\varepsilon, \delta}(t,x)\diamond \dot W^{\varepsilon, \delta}(t,x)=\int_0^t\int_{\R^d}g_\delta(t-s)p_\varepsilon(x-y)u^{\varepsilon, \delta}(t,x) W^\diamond(ds,dy).
\end{equation*}
Therefore, the mild solution to (\ref{appspde'}) is, as defined in \cite{MR2449130}, an adapted random field $\{u^{\varepsilon, \delta}(t,x), t\ge0, x\in \R^d\}$ which is square integrable for all fixed $(t,x)$ and satisfies the following integral equation,
\begin{align*}
 &u^{\varepsilon, \delta}(t,x)=Q_tu_0(x)+\int_0^t\int_{\R^d}q_{t-s}(x-y)u^{\varepsilon, \delta}(s,y)\diamond \dot W^{\varepsilon,\delta}(s,y) dsdy\\
 &=Q_tu_0(x)+\int_0^t\int_{\R^d}\left(\int_0^t\int_{\R^d}q_{t-s}(x-y)g_\delta(s-r)p_\varepsilon(y-z)u^{\varepsilon, \delta}(s,y) dsdy\right) W^{\diamond}(dr,dz).
\end{align*}
Denote \[Z_{t,x}^{\varepsilon, \delta}(r,z)=\int_0^t\int_{\R^d}q_{t-s}(x-y)g_\delta(s-r)p_\varepsilon(y-z)u^{\varepsilon, \delta}(s,y) dsdy.\]
Thus to show that an adapted and square integrable process $\{u^{\varepsilon,\delta}(t,x), t\ge0, x\in \R^d\}$ is a mild solution to (\ref{appspde'}), it is equivalent to show $u^{\varepsilon, \delta}(t,x)=Q_tu_0(x)+\delta(Z_{t,x}^{\varepsilon, \delta})$. Therefore by the definition of the divergence operator $\delta$, it is equivalent to show that for any $F\in \mathbb D^{1,2}$ with mean zero, 
\begin{equation}\label{e5.8}
 \E[Fu^{\varepsilon,\delta}(t,x)]=\E[\langle Z_{t,x}^{\varepsilon, \delta}, DF\rangle_\mathcal H].
\end{equation}
Let 
\begin{equation}\label{e5.9}
 u^{\varepsilon,\delta}(t,x)=\E^X\left[u_0(X_t^x)\exp\left(W(\Phi_{t,x}^{\varepsilon, \delta})-\frac12\|\Phi_{t,x}^{\varepsilon, \delta}\|^2_\mathcal H\right)\right],
\end{equation}
where $\Phi_{t,x}^{\varepsilon, \delta}$ is given by (\ref{Phi}). Using a similar argument based on the technique of $S$-transform as in the proof of \cite[Proposition 5.2]{MR2449130}, we can show that $u^{\varepsilon,\delta}(t,x)$ given by (\ref{e5.9}) satisfies (\ref{e5.8}), and hence it is a mild solution to (\ref{appspde'}).

{\bf Step 2.} In this step, we will show that 
\begin{equation}\label{e5.10}
\lim_{\varepsilon,\delta\downarrow 0} \E\left[\left(u^{\varepsilon, \delta}(t,x)\right)^p\right]=\E\left[\prod_{j=1}^p u_0(X_t^j+x)\exp\left(\sum_{1\le j<k\le p} \int_0^t\int_0^t |r-s|^{-\beta_0}\gamma(X_r^j-X_s^k)drds\right)\right].
\end{equation}

Recall that $u_0$ is a bounded and continuous function, and without loss of generality we assume $u_0(x)\equiv 1$ for a simpler exposition. Denote \[\Phi_{t,x}^{\varepsilon, \delta, j}(r,y):=\int_0^t g_{\delta}(t-s-r)p_\varepsilon(X_s^j+x-y)ds\cdot I_{[0,t]}(r), ~ j=1, \dots, p.\] 
The $p$-moment of $u^{\varepsilon, \delta}(t,x)$ is
\begin{align*}
 &\E\left[\left(u^{\varepsilon, \delta}(t,x)\right)^p\right]=\E^W\E^X \prod_{j=1}^p \exp\left(W(\Phi_{t,x}^{\varepsilon, \delta, j})-\frac12\|\Phi_{t,x}^{\varepsilon, \delta, j}\|_{\mathcal H}^2\right)\\
 &=\E^X\exp\left(\frac12\|\sum_{j=1}^p\Phi_{t,x}^{\varepsilon, \delta, j}\|^2_\mathcal H-\frac12\sum_{j=1}^p\|\Phi_{t,x}^{\varepsilon, \delta, j}\|_{\mathcal H}^2\right)=\E^X\exp\left(\sum_{1\le i<j\le p}\langle \Phi_{t,x}^{\varepsilon, \delta, j},\Phi_{t,x}^{\varepsilon, \delta, k}\rangle_{\mathcal H}\right).
\end{align*}
As in the proof of Theorem \ref{appfeyn},  we can show  that \begin{align*}
\langle \Phi_{t,x}^{\varepsilon,\delta,j}, \Phi_{t,x}^{\varepsilon,\delta,k}\rangle_\mathcal H
 =&\frac{1}{(2\pi)^d} \int_{[0,t]^4} \int_{\R^d} (\widehat p_\varepsilon(\xi))^2\exp\big(-i\xi\cdot (X_{s_1}^j-X_{s_2}^k)\big)\\
&\qquad g_{\delta}(t-s_1-r_1)g_{\delta}(t-s_2-r_2)
|r_1-r_2|^{-\beta_0}\mu(d\xi)dr_1dr_2ds_1ds_2,
\end{align*}
and that $\langle \Phi_{t,x}^{\varepsilon,\delta,j}, \Phi_{t,x}^{\varepsilon,\delta,k}\rangle_\mathcal H
$ converges to $\int_{[0,t]^2}|s_1-s_2|^{-\beta_0}\gamma(X_{s_1}^j-X_{s_2}^k)ds_1ds_2$ in $L^1$ as $(\varepsilon, \delta)$ tends to zero. Now to prove the equality (\ref{e5.10}), it suffices to show that for any $\lambda >0$,
\begin{equation}\label{e5.11}
 \sup_{\varepsilon, \delta >0}\E\left[\exp\left(\lambda \langle \Phi_{t,x}^{\varepsilon,\delta,j}, \Phi_{t,x}^{\varepsilon,\delta,k} \rangle_\mathcal H\right)\right]<\infty.
\end{equation}
By (\ref{inner'}) and (\ref{gdelta}), there exists a positive constant $C$ depending on $\beta_0$ only such that
\begin{align*}
 \langle \Phi_{t,x}^{\varepsilon,\delta,j}, \Phi_{t,x}^{\varepsilon,\delta,k}\rangle_\mathcal H\le C\int_{[0,t]^2}\int_{\R^d}(\widehat p_\varepsilon(\xi))^2\exp(-i\xi\cdot(X_r^j-X_s^k))|r-s|^{-\beta_0}\mu(d\xi)drds.
\end{align*}
Hence to obtain (\ref{e5.11}), it is sufficient to prove  that for any $\lambda>0$,
\begin{equation}
\sup_{\varepsilon>0}\E\left[\exp\left(\lambda \int_{[0,t]^2}\int_{\R^d}(\widehat p_\varepsilon(\xi))^2\exp(-i\xi\cdot(X_r-\widetilde X_s))|r-s|^{-\beta_0}\mu(d\xi)drds\right)\right]<\infty,
\end{equation}
where $\widetilde X$ is an independent copy of $X$. For the $n$-th moment of the exponent, similar to the proof of Theorem \ref{thmfk}, we have that for any $\varepsilon>0$,
\begin{align*}
 &\E\left[\left(\int_{[0,t]^2}\int_{\R^d}(\widehat p_\varepsilon(\xi))^2\exp(-i\xi\cdot(X_r-\widetilde X_s))|r-s|^{-\beta_0}\mu(d\xi)drds\right)^n\right]\\
 =& \int_{[0,t]^{2n}}\int_{\R^{nd}}\prod_{j=1}^n (\widehat p_\varepsilon(\xi_j))^2E\exp(-i\sum_{j=1}^n\xi_j\cdot(X_{r_j}-\widetilde X_{s_j}))\prod_{j=1}^n|r_j-s_j|^{-\beta_0}\mu(d\xi)drds\\
 \le& \int_{[0,t]^{2n}}\int_{\R^{nd}}\E\exp(-i\sum_{j=1}^n\xi_j\cdot(X_{r_j}-\widetilde X_{s_j}))\prod_{j=1}^n|r_j-s_j|^{-\beta_0}\mu(d\xi)drds\\
 =&\E\left[\left(\int_0^t \int_0^t |r-s|^{-\beta_0}\gamma(X_r-\widetilde X_s)drds\right)^n\right].
\end{align*}
Now to prove (\ref{e5.11}), it is sufficient to prove that for any $\lambda >0$, 
\begin{equation}\label{e5.14'}
\E\left[\exp\left(\lambda\int_0^t \int_0^t |r-s|^{-\beta_0}\gamma(X_r-\widetilde X_s)drds\right)\right]<\infty.
\end{equation}
Note that 
\begin{align*}
&\E\left[\left(\lambda\int_0^t \int_0^t |r-s|^{-\beta_0}\gamma(X_r-\widetilde X_s)drds\right)^n\right]=\lambda^n\int_{[0,t]^{2n}}\prod_{j=1}^n|r_j-s_j|^{-\beta_0}\E\left[\prod_{j=1}^n\gamma(X_{r_j}-\widetilde X_{s_j})\right]drds\\
&=\lambda^n (n!)^2 \int_{[0,t]^{2n}}\int_{\R^{2nd}} h_n(s,y,t,0)h_n(r,z, t,0)\prod_{j=1}^n |s_j-r_j|^{-\beta_0}\prod_{j=1}^n \gamma(y_j-z_j) dydz drds,
\end{align*}
where $h_n$ is given by (\ref{hn}), and the last equality is obtained by using the independent increment property of $X$. Then  (\ref{e5.14'}) can be obtained as in the proof  of Theorem \ref{thmfkskr}. 

{\bf Step 3.} As in the proof of Theorem \ref{thmfk}, we can  show that $\sup_{\varepsilon, \delta>0}\sup_{t\in[0,T], x\in\R^d}\E[|u^{\varepsilon,\delta}(t,x)|^p]<\infty$, $u^{\varepsilon,\delta}(t,x)$ converges to a limit  denoted by $u(t,x)$ in $L^p$ for any $p>0$ as $(\varepsilon, \delta)$ goes to zero, and moreover, $u(t,x)$  satisfies the formula (\ref{momskr}). Therefore, by the uniqueness of the mild solution to (\ref{spde'}), to conclude the proof, we only need to show that $u(t,x)$ is a mild solution to (\ref{spde'}), i.e.,
\begin{equation}\label{e5.12}
 \E[Fu(t,x)]=\E[\langle Z_{t,x}, DF\rangle_\mathcal H],
\end{equation}
for any $F\in \mathbb D^{1,2}$ with $\E[F]=0$, where $Z_{t,x}(r,z)=q_{t-r}(x-z)u(r,z)$.

In a way similar to the proof of Theorem \ref{appfeyn}, we can prove that $\lim_{\varepsilon,\delta\downarrow 0}\E[\|Z_{t,x}^{\varepsilon,\delta}-Z_{t,x}\|_{\mathcal H}^2]=0.$ Then we can show the equality (\ref{e5.12}) by  letting $(\varepsilon,\delta)$ in  (\ref{e5.8}) go to zero,  noting that $F\in \mathbb D^{1,2}$ and $\lim_{\varepsilon,\delta\downarrow 0}u^{\varepsilon, \delta}(t,x)=u(t,x)$ in $L^2$. \hfill 
\end{proof}
\begin{remark}
In the second step of the proof, actually we proved that under  Hypothesis (II), (\ref{e5.14'}) holds, i.e., for any $\lambda>0$
\begin{equation*}\label{e5.14}
\E\left[\exp\left(\lambda\int_0^t \int_0^t |r-s|^{-\beta_0}\gamma(X_r-\widetilde X_s)drds\right)\right]<\infty.
\end{equation*}
\end{remark}

\subsection{H\"older continuity}

\begin{hypothesis}[S2]  The spectral measure $\mu$ satisfies that for all $a\in \R^d$, there exist $\alpha_1\in(0,1]$ and $C>0$ such that
\[\sup_{z\in\R^d}\int_0^T\int_{\R^d}e^{-s\Psi(\xi+z)}\left(1-e^{-i(\xi+z)\cdot a}\right)\mu(d\xi)ds\le C|a|^{2\alpha_1}.\]
\end{hypothesis}
\begin{hypothesis}[T2]  The spectral measure $\mu$ satisfies, for some $\alpha_2\in (0,1),$ 
\[\int_{\R^d}\frac{(\Psi(\xi))^{\alpha_2}}{1+\Psi(\xi)}\mu(d\xi)<\infty.\]
\end{hypothesis}

\begin{remark}\label{rm5.7} Similar to the Stratonovich case, we have the following sufficient condition for  Hypothesis (S2) to hold:
\begin{equation}\label{s2'}
\sup_{z\in\R^d}\int_{\R^d}\frac{|\xi+z|^{2\alpha_1}}{1+\Psi(\xi+z)}\mu(d\xi)<\infty.
\end{equation}
Furthermore, if $\eta(\xi):=\Psi(\xi)/|\xi|^{2\alpha_1}$ is a L\'evy characteristic exponent (which is equivalent to say that $-\eta(\xi)$ is continuous, conditionally positive definite and $\eta(0)=0$, see, e.g., \cite[Theorem 1.2.17]{MR2512800}; a special case in which $\eta(\xi)$ is the characteristic exponent of a symmetric stable process is that $\Psi(\xi)=|\xi|^\alpha$ with $\alpha>2\alpha_1$), then condition (\ref{s2'}) is equivalent to 
\begin{equation}\label{s2''}
\int_{\R^d}\frac{|\xi|^{2\alpha_1}}{1+\Psi(\xi)}\mu(d\xi)<\infty.
\end{equation}
Clearly (\ref{s2'}) implies (\ref{s2''}). Now we show that the inverse is true.
Let $M$ be a positive number such that $\eta(\xi)\ge 1$ for all $|\xi|\ge M$. Noting that 
$\sup_{z\in\R^d}\mu([|\xi+z|\le M])<\infty$ by Lemma \ref{lem2.3measure}. 
 \begin{align*}
&\int_{\R^d}\frac{|\xi+z|^{2\alpha_1}}{1+\Psi(\xi+z)}\mu(d\xi)=\int_{[|\xi+z|\le M]}\frac{|\xi+z|^{2\alpha_1}}{1+\Psi(\xi+z)}\mu(d\xi)+\int_{[|\xi+z|> M]}\frac{|\xi+z|^{2\alpha_1}}{1+\Psi(\xi+z)}\mu(d\xi)\\
\le & M^{2\alpha_1} \sup_{z\in\R^d}\mu([|\xi+z|\le M])+2\int_{\R^d}\frac1{1+\eta(\xi+z)}\mu(d\xi)\\
=&C+2 \int_{\R^d}\int_0^\infty e^{-t}e^{-t\eta(\xi+z)}dt \mu(d\xi)\le C+2 \int_{\R^d}\int_0^\infty e^{-t}e^{-t\eta(\xi)}dt \mu(d\xi) \text{ (Lemma \ref{lemma})}\\
=& C+2 \int_{\R^d}\frac{1}{1+\eta(\xi)}\mu(d\xi)\le D+2 \int_{\R^d}\frac{|\xi|^{2\alpha_1}}{1+\Psi(\xi)}\mu(d\xi),
\end{align*}
where $D$ is another constant that may be different from $C$.

Similarly, Hypothesis (T2) actually implies and hence is equivalent to the condition
\begin{equation}\label{t2'}
\sup_{z\in \R^d}\int_{\R^d}\frac{(\Psi(\xi+z))^{\alpha_2}}{1+\Psi(\xi+z)}\mu(d\xi)<\infty.
\end{equation}
Indeed, for all $z\in \R^d$,
\begin{align*}
&\int_{\R^d}\frac{(\Psi(\xi+z))^{\alpha_2}}{1+\Psi(\xi+z)}\mu(d\xi)\le  \int_{\R^d}\left(\frac{1}{1+\Psi(\xi+z)}\right)^{1-\alpha_2}\mu(d\xi)\\
&=\int_{\R^d}~\frac1{\Gamma(1-\alpha_2)}\int_0^\infty t^{-\alpha_2} e^{-[1+\Psi(\xi+z)]t}dt ~\mu(d\xi)\\
&\le \int_{\R^d}~\frac1{\Gamma(1-\alpha_2)}\int_0^\infty t^{-\alpha_2} e^{-[1+\Psi(\xi)]t}dt ~\mu(d\xi) \text{ (Lemma \ref{lemma4.6})}\\&= \int_{\R^d}\left(\frac{1}{1+\Psi(\xi)}\right)^{1-\alpha_2}\mu(d\xi),
\end{align*}
where the first equality follows from the formula
$c^{-\alpha}=\frac{1}{\Gamma(\alpha)} \int_0^\infty t^{\alpha-1} e^{-ct}dt$  for $c>0$ and  $\alpha\in(0,1)$.
Finally Hypothesis (T2) implies (\ref{t2'}) because of the following equivalence 
\[\int_{\R^d}\frac{(\Psi(\xi))^{\alpha_2}}{1+\Psi(\xi)}\mu(d\xi)<\infty\Longleftrightarrow \int_{\R^d}\left(\frac{1}{1+\Psi(\xi)}\right)^{1-\alpha_2}\mu(d\xi)<\infty\]
which is due to the facts $\lim_{|\xi|\to\infty}\Psi(\xi)=\infty$ and $\mu(A)<\infty$ for bounded $A\in \mathcal B(\R^d)$,

\end{remark}
\begin{theorem}\label{thmholder'} Let $u_0(x)\equiv1$ and $u(t,x)$ be the unique mild solution to (\ref{spde'}).  If  $\mu$ satisfies  Hypothesis (S2), then  $u(t,x)$ has a version that is $\theta_1$-H\"older continuous in $x$ with $\theta_1<\alpha_1$ on any compact set of $[0,\infty)\times \R^d$;  Similarly, if $\mu$ satisfies Hypothesis (T2), the solution $u(t,x)$ has a version that is $\theta_2$-H\"older continuous in $t$ with $\theta_2<[\alpha_2\wedge(1-\beta_0)]/2$ on any compact set of $[0,\infty)\times \R^d$.
\end{theorem}

\begin{remark} As in Remark \ref{RemarkHolder}, we apply the above result to the case when $\L=-(-\Delta)^{\alpha/2}$ with  $\alpha\in (0,2]$,  and  $\gamma(x)=|x|^{-\beta}, \beta\in(0,d)$ or $\gamma(x)=\prod_{j=1}^d |x_j|^{-\beta_j}, \beta_j\in(0,1), j=1,\dots, d$.  Since condition \eqref{s2''} is equivalent to $\alpha_1<\frac12(\alpha-\beta)$ and Hyperthesis (T2) is equivalent to $\alpha_2<1-\frac\beta\alpha,$ if we assume condition \eqref{s2''} and Hypothesis (T2), the solution $u(t,x)$ has a version that is $\theta_1$-H\"older continuous in $x$ with $\theta_1\in (0, \frac12(\alpha-\beta))$ and $\theta_2$-H\"older continuous in $t$ with $\theta_2\in(0, \frac12(1-\frac\beta\alpha)\wedge (1-\beta_0))$, on any compact set of  $[0,\infty)\times \R^d$.
\end{remark}

\begin{proof}
Let $u(t,x)=1+\sum_{n=1}^\infty I_n(h_n(\cdot, t,x))$ and $u(s,y)=1+\sum_{n=1}^\infty I_n(h_n(\cdot,s,y))$, where $h_n$ is given by (\ref{hn}). Then for $p>2$,
\begin{align}\label{lp}
&\|u(t,x)-u(s,y)\|_{L^p}\le \sum_{n=1}^\infty \|I_n(h_n(\cdot,t,x))-I_n(h_n(\cdot,s,y))\|_{L^p}\notag\\
&\le \sum_{n=1}^\infty (p-1)^{n/2}\|I_n(h_n(\cdot,t,x))-I_n(h_n(\cdot,s,y))\|_{L^2} \notag\\
&= \sum_{n=1}^\infty (p-1)^{n/2}\sqrt{n!} \|h_n(\cdot,t,x)-h_n(\cdot,s,y)\|_{\mathcal H^{\otimes n}},
\end{align} 
where the last inequality holds due to the equivalence of $L^p$ norms for $p>1$ on any Wiener chaos space $\mathbb H_n$ (\cite[Theorem 1.4.1]{MR2200233}), and the last equality follows from (\ref{e2.10}).

{\bf Step 1.} First, we study the spatial continuity. Suppose that $s=t$, similar as in the proof of Theorem \ref{thmfkskr}, we have 
\begin{align*}
&n!\|h_n(\cdot,t,x)-h_n(\cdot,t,y)\|_{\mathcal H^{\otimes n}}^2\\
=&n!\left(\|h_n(\cdot,t,x)\|^2_{\mathcal H^{\otimes n}}+\|h_n(\cdot,t,y)\|^2_{\mathcal H^{\otimes n}}-2\langle h_n(\cdot,t,x), h_n(\cdot,s,y)\rangle_{\mathcal H^{\otimes n}}\right)\\
= &\frac{2}{n!}\int_{[0,t]^{2n}}\int_{\R^{nd}} \left[1- e^{-i(x-y)\cdot (\xi_1+\dots+\xi_n)}\right]\prod_{j=1}^n \exp\Big[-[r_{\sigma(j+1)}-r_{\sigma(j)}]\Psi(\xi_{\sigma(1)}+\dots+\xi_{\sigma(j)})\Big]\\
& ~~~~~\prod_{j=1}^n \exp\Big[-[s_{\eta(j+1)}-s_{\eta(j)}]\Psi(\xi_{\eta(1)}+\dots+\xi_{\eta(j)})\Big]\mu(d\xi)\prod_{j=1}^n|r_j-s_j|^{-\beta_0} drds 
\end{align*}
where $\sigma$ and $\eta$ are permutations of the set $\{1,2,\dots, n\}$ such that $r_{\sigma(1)}<r_{\sigma(2)}<\dots<r_{\sigma(n)}$ and $s_{\eta(1)}<s_{\eta(2)}<\dots<s_{\eta(n)}$. Denote
\[A^2(r)=\int_{\R^{nd}}\left[1- e^{-i(x-y)\cdot (\xi_1+\dots+\xi_n)}\right]\prod_{j=1}^n \exp\Big[-2[r_{\sigma(j+1)}-r_{\sigma(j)}]\Psi(\xi_{\sigma(1)}+\dots+\xi_{\sigma(j)})\Big]\mu(d\xi).\]
Recall the notations $D_t=2\int_0^t s^{-\beta_0}ds$ and $\Omega_t^n=\left\{(s_1, \dots, s_n)\in [0,\infty)^n: \sum_{j=1}^n s_j\le t\right\}$. We have 
\begin{align*}
&n!\|h_n(\cdot,t,x)-h_n(\cdot,t,y)\|_{\mathcal H^{\otimes n}}^2
\le \frac{2}{n!} \int_{[0,t]^{2n}}   A^2(r) \prod_{j=1}^n|s_j-r_j|^{-\beta_0} dsdr\le \frac{2}{n!}D_t^n \int_{[0,t]^n} A^2(r)dr\\
=& 2D_t^n\int_{[0<r_1<r_2<\dots<r_n<t]} \int_{\R^{nd}}\left[1- e^{-i(x-y)\cdot (\xi_1+\dots+\xi_n)}\right]\prod_{j=1}^n \exp\Big[-2[r_{j+1}-r_{j}]\Psi(\xi_{1}+\dots+\xi_{j})\Big]\mu(d\xi)dr\\
=&2D_t^n\int_{\Omega_t^n}\int_{\R^{nd}} \left[1- e^{-i(x-y)\cdot (\xi_1+\dots+\xi_n)}\right]\prod_{j=1}^n \exp\Big[-2s_j\Psi(\xi_{1}+\dots+\xi_{j})\Big]\mu(d\xi)ds\\
\le& 2D_t^n \sup_{z\in\R^d}\int_0^t \int_{\R^d} \left[1- e^{-i(x-y)\cdot (z+\xi_n)}\right] \exp[-2s_n \Psi(z+\xi_n)]\mu(d\xi_n)ds_n\\
&~~~~~~~~~~~~~~~~\times \int_{\Omega_t^{n-1}}\int_{\R^{(n-1)d}} \prod_{j=1}^{n-1} \exp\left[-2s_j\Psi(\xi_{1}+\dots+\xi_{j})\right]\mu(d\xi_1)\dots\mu(d\xi_{n-1})ds_1\dots ds_{n-1}\\
\le& C D_t^n |x-y|^{2\alpha_1} \int_{\Omega_t^{n-1}}\int_{\R^{(n-1)d}} \prod_{j=1}^{n-1} \exp\left[-2s_j\Psi(\xi_{j})\right]\mu(d\xi)ds.~~~ \text{ (By Hypothesis (S2))}
\end{align*}
 Applying Lemma (\ref{hhnt}), we have 
\[\sqrt{n!}\|h_n(\cdot,t,x)-h_n(\cdot,t,y)\|_{\mathcal H^{\otimes n}}\le |x-y|^{\alpha_1}C D_t^{n/2} \sum_{k=0}^{n-1} \sqrt{\binom {n-1}{k} \frac{t^k}{k!}m_N^{k}\left[A_0\varepsilon_N\right]^{n-1-k}}.\]
As in the proof of Theorem \ref{thmfkskr}, we can choose $N$ large enough, such that $$\sum_{n=1}^\infty D_t^{n/2} \sum_{k=0}^n \sqrt{\binom nk \frac{t^k}{k!}m_N^{k}\left[A_0\varepsilon_N\right]^{n-k}}<\infty,$$
and hence there exists a constant $C$ such that 
\[\|u(t,x)-u(t,y)\|_{L^p}\le C |x-y|^{\alpha_1},\]
which implies the spatial H\"older continuity of $u(t,x)$.

{\bf Step 2.} Now we consider the H\"older continuity in time, assuming that $0\le s<t\le T$ and $x=y$. Then for the estimation on the $n$-th chaos space, we have
\begin{align*}
&n!\|h_n(\cdot,t,x)-h_n(\cdot,s,x)\|_{\mathcal H^{\otimes n}}^2\\
=&n!\left(\|h_n(\cdot,t,x)\|^2_{\mathcal H^{\otimes n}}+\|h_n(\cdot,s,x)\|^2_{\mathcal H^{\otimes n}}-2\langle h_n(\cdot,t,x), h_n(\cdot,s,x)\rangle_{\mathcal H^{\otimes n}}\right)\\
=&n! \Bigg[\int_{[0,t]^{2n}}\int_{\R^{nd}} \mathcal F h_n(u, \cdot, t,x)(\xi)\overline{\mathcal Fh_n(v,\cdot, t,x)(\xi)}\mu(d\xi)\prod_{j=1}^n|u_j-v_j|^{-\beta_0}dvdu\\
&~~~~~~+\int_{[0,s]^{2n}}\int_{\R^{nd}} \mathcal F h_n(u, \cdot, s,x)(\xi)\overline{\mathcal Fh_n(v,\cdot, s,x)(\xi)}\mu(d\xi)\prod_{j=1}^n|u_j-v_j|^{-\beta_0} dvdu\\
&~~~~~~-2\int_{[0,t]^{n}\times[0,s]^n}\int_{\R^{nd}} \mathcal F h_n(u, \cdot, t,x)(\xi)\overline{\mathcal Fh_n(v,\cdot, s,x)(\xi)}\mu(d\xi)\prod_{j=1}^n|u_j-v_j|^{-\beta_0}dvdu\Bigg].
\end{align*}
Therefore
\begin{equation}\label{dn}
n!\|h_n(\cdot,t,x)-h_n(\cdot,s,x)\|_{\mathcal H^{\otimes n}}^2\le n!(D_n+D_n'),
\end{equation}
where 
\begin{align*}
D_n=&\int_{[0,t]^{2n}}\int_{\R^{nd}} \mathcal F h_n(u, \cdot, t,x)(\xi)\overline{\mathcal Fh_n(v,\cdot, t,x)(\xi)}\mu(d\xi)\prod_{j=1}^n|u_j-v_j|^{-\beta_0}dvdu\\
&~~~~~~-\int_{[0,t]^{n}\times[0,s]^n}\int_{\R^{nd}} \mathcal F h_n(u, \cdot, t,x)(\xi)\overline{\mathcal Fh_n(v,\cdot, s,x)(\xi)}\mu(d\xi)\prod_{j=1}^n|u_j-v_j|^{-\beta_0}dvdu,
\end{align*}
and 
\begin{align*}
D_n'=&\int_{[0,t]^{n}\times[0,s]^n}\int_{\R^{nd}} \mathcal F h_n(u, \cdot, t,x)(\xi)\overline{\mathcal Fh_n(v,\cdot, t,x)(\xi)}\mu(d\xi)\prod_{j=1}^n|u_j-v_j|^{-\beta_0}dvdu\\
&~~~~~~-\int_{[0,s]^{2n}}\int_{\R^{nd}} \mathcal F h_n(u, \cdot, t,x)(\xi)\overline{\mathcal Fh_n(v,\cdot, s,x)(\xi)}\mu(d\xi)\prod_{j=1}^n|u_j-v_j|^{-\beta_0}dvdu.
\end{align*}
We will just estimate $D_n$, and $D_n'$ will share the same upper bound of $D_n$.

Clearly, $D_n=A_n+B_n$
where 
\begin{equation}\label{an}
A_n=\int_{[0,t]^{n}\times ([0,t]^n\backslash [0,s]^n)}\int_{\R^{nd}} \mathcal F h_n(u, \cdot, t,x)(\xi)\overline{\mathcal Fh_n(v,\cdot, t,x)(\xi)}\mu(d\xi)\prod_{j=1}^n|u_j-v_j|^{-\beta_0}dvdu
\end{equation}
and 
\begin{align}
B_n=&\int_{[0,t]^{n}\times[0,s]^n}\int_{\R^{nd}} \left(\overline{\mathcal Fh_n(v,\cdot, t,x)(\xi)}-\overline{\mathcal Fh_n(v,\cdot, s,x)(\xi)}\right)
\notag\\
&~~~~~~~~~~~~~~~~~~~~~\mathcal F h_n(u, \cdot, t,x)(\xi)\mu(d\xi)\prod_{j=1}^n|u_j-v_j|^{-\beta_0}dvdu. \label{bn}
\end{align}
To get an estimation for the right-hand side of (\ref{dn}), we will separate the rest of the proof into three parts for easier reading. 

{\bf Step 2(a).} In this part, we will estimate $A_n$ given in (\ref{an}). Note that $[0,t]^n=\cup_{k_j\in\{0,1\}} I_{k_1}\times I_{k_2}\times \dots \times I_{k_n}$ with $I_1=[0,s]$ and $I_2=[s,t]$. Hence $[0,t]^n\backslash [0,s]^n$ is the union of $2^n-1$ disjoint interval products, each of which contains at least one $[s,t]$.   Denote $E_{n,j}$ the product of $n$ intervals, all of which are $[0,t]$  except that the $j$-th interval is $[s,t]$. Therefore, for the term $A_n$, we have
\begin{align}
 A_n\le& 2^n \sup_{j=1,\dots, n}\int_{[0,t]^n\times E_{n,j}}\int_{\R^{nd}} \mathcal F h_n(u, \cdot, t,x)(\xi)\overline{\mathcal Fh_n(v,\cdot, t,x)(\xi)}\mu(d\xi)\prod_{j=1}^n|u_j-v_j|^{-\beta_0}dvdu\notag\\
 \le & 2^n \sup_{j=1,\dots, n}\int_{[0,t]^n\times E_{n,j}}\left (A^2_{t,x}(u) +A^2_{t,x}(v)\right)\prod_{j=1}^n|u_j-v_j|^{-\beta_0}dvdu \label{ean}
 \end{align}
with $A_{t,x}(u)$ given in (\ref{atx}). Denoting $D_t=2\int_0^t |s|^{-\beta_0}ds$, for positive function $f$, we have the following estimates
\[\int_0^t\int_0^t f(u)|u-v|^{-\beta_0}dvdu\le D_t \int_0^t f(u) du,\]
\[\int_0^t\int_s^t f(u)|u-v|^{-\beta_0}dvdu\le \frac{2^{\beta_0}}{1-\beta_0}(t-s)^{1-\beta_0} \int_0^t f(u) du,\]
and
\[\int_0^t\int_s^t f(v)|u-v|^{-\beta_0}dvdu\le D_t \int_s^t f(v) dv.\]
Applying those estimates, we get
\begin{align}
&\int_{[0,t]^n\times E_{n,j}}\left (A^2_{t,x}(u) +A^2_{t,x}(v)\right)\prod_{j=1}^n|u_j-v_j|^{-\beta_0}dvdu\notag\\
\le & \frac{2^{\beta_0}}{1-\beta_0}(t-s)^{1-\beta_0}D_t^{n-1}\int_{[0,t]^n} A^2_{t,x}(u)du+D_t^n \int_{E_{n,j}} A^2_{t,x}(v)dv.\label{e.5.20}
\end{align}
Note that Hypothesis (T2) implies 
\begin{equation}\label{e5.24}
\int_{\R^d} \frac{1}{1+(\Psi(\xi))^{1-\alpha_2}}\mu(d\xi)<\infty,
\end{equation}
and hence there exists $C>0$ depending on the measure $\mu$ and $\alpha_2$ such that for all $x>0$
\[\int_{\R^d} e^{-x\Psi(\xi)}\mu(d\xi)\le C(1+x^{\alpha_2-1}) \]
by Lemma \ref{lemma1}. On the other hand, by Lemma \ref{multii}, we have 
\begin{align}
&\int_{[0<v_1<v_2<\dots<v_n<t]} \int_{\R^{nd}}\prod_{j=1}^n (1+ (v_{j+1}-v_j)^{\alpha_2-1})\mu(d\xi)dv\notag\\
&\le C^n\sum_{\tau\in\{0,1\}^n}\frac{\prod_{j=1}^n \Gamma(\tau_j(\alpha_2-1)+1)}{\Gamma(\sum_{j=1}^n \tau_j(\alpha_2-1)+n+1)}t^{\sum_{j=1}^n \tau_j(\alpha_2-1)+n}.\label{e5.25}
\end{align}
Combining (\ref{e5.24}) and (\ref{e5.25}) and using the approach in Remark \ref{remark3.8}, we have for $t\in[0,T]$ with $T\ge1$,
\begin{align}\label{e5.26'}
\int_{[0,t]^n} A^2_{t,x}(u)du\le \frac{C^n}{n!}\sum_{m=0}^n \binom{n}{m} \frac{t^{m(\alpha_2-1)+n}}{\Gamma(m(\alpha_2-1)+n+1)}\le \frac{(2C)^n}{n!} \frac{T^n}{\Gamma(n\alpha_2+1)}.
\end{align}
Similarly, for all $j\in\{1,2,\dots, n\}$ and $0\le s<t\le T$ with $T\ge1$, we have
\begin{align}
 &\int_{E_{n,j}} A^2_{t,x}(v)dv\le \frac1{(n!)^2}\int_{[0,t]^{n-1}\times [s,t]}\int_{\R^{nd}} \prod_{j=1}^n \exp\left(-2(v_{\sigma(j+1)}-v_{\sigma(j)})\psi(\xi_{\sigma(j)})\right) \mu(d\xi)dv \notag\\
 &\le  \frac{ C^n}{(n!)^2}\int_{[0,t]^{n-1}\times [s,t]}\prod_{j=1}^n \left(1+(v_{\sigma(j+1)}-v_{\sigma(j)})^{\alpha_2-1}\right)dv \notag\\
  &\le \frac{ C^n}{(n!)^2}\left(\int_{[0,t]^{n}}-\int_{[0,s]^{n}}\right)\prod_{j=1}^n \left(1+(v_{\sigma(j+1)}-v_{\sigma(j)})^{\alpha_2-1}\right)dv \notag\\
 &=\frac{ C^n}{n!} \left(\int_{[0<v_1<\dots<v_n<t]}-\int_{[0<v_1<\dots<v_n<s]^{n}}\right)\prod_{j=1}^n \left(1+(v_{j+1}-v_{j})^{\alpha_2-1}\right)dv\notag\\
 &=\frac{ C^n}{n!}\sum_{\tau\in\{0,1\}^n}\frac{\prod_{j=1}^n \Gamma(\tau_j(\alpha_2-1)+1)}{\Gamma(\sum_{j=1}^n \tau_j(\alpha_2-1)+n+1)}(t^{\sum_{j=1}^n \tau_j(\alpha_2-1)+n}-s^{\sum_{j=1}^n \tau_j(\alpha_2-1)+n})\notag\\
&\le \frac1{n!} \frac{C^n}{\Gamma(n\alpha_2+1)} nT^n(t-s)^{\alpha_2}. \label{e5.27}
 \end{align}
 The last inequality holds because $t^{\sum_{j=1}^n \tau_j(\alpha_2-1)+n}-s^{\sum_{j=1}^n \tau_j(\alpha_2-1)+n}\le nT^n(t-s)^{\alpha_2}$ for all $n$ and $\tau.$
Combining the above (\ref{e5.26'}) and (\ref{e5.27}) with (\ref{ean}) and (\ref{e.5.20}), we have 
\begin{align}\label{estan}
 A_n &\le \frac1{n!} \frac{C^n}{\Gamma(n\alpha_2+1)} \left((t-s)^{1-\beta_0}+(t-s)^{\alpha_2}\right),
 \end{align}
where $C$ depends on the measure $\mu$, $T, \beta_0$ and $\alpha_2.$
 
 {\bf Step 2(b).}  The term  $B_n$ given in (\ref{bn}) will be estimated in this part. 
 \begin{align*}
  B_n\le & \frac1{(n!)^2} \int_{[0,t]^n\times[0,s]^n}\int_{\R^{nd}} \left|e^{-(t-v_{\sigma(n)})\Psi(\xi_1+\dots+\xi_n)}-e^{-(s-v_{\sigma(n)})\Psi(\xi_1+\dots+\xi_n)}\right|\\
  & ~~~\prod_{j=1}^{n-1} e^{-(v_{\sigma(j+1)}-v_{\sigma(j)})\Psi(\xi_{\sigma(1)}+\dots+\xi_{\sigma(j)})} \mathcal F h_n(u,\dots, t,x)(\xi)\mu(d\xi) \prod_{j=1}^n|u_j-v_j|^{-\beta_0}dvdu\\
  \le & 2(t-s)^{\alpha_2} \frac1{(n!)^2} \int_{[0,t]^n\times[0,s]^n}\int_{\R^{nd}} \left(\Psi(\xi_1+\dots+\xi_n)\right)^{\alpha_2} \prod_{j=1}^{n} e^{-(v_{\sigma(j+1)}-v_{\sigma(j)})\Psi(\xi_{\sigma(1)}+\dots+\xi_{\sigma(j)})}\\
  &~~~~~~~~\prod_{j=1}^{n} e^{-(u_{\eta(j+1)}-u_{\eta(j)})\Psi(\xi_{\eta(1)}+\dots+\xi_{\eta(j)})} \mu(d\xi) \prod_{j=1}^n|u_j-v_j|^{-\beta_0}dvdu, 
 \end{align*}
where $v_{n+1}=s, u_{n+1}=t$ and $\sigma$  and $\eta$ are permutations such that $0<v_{\sigma(1)}<\dots<v_{\sigma(n)}<t$ and $0<u_{\eta(1)}<\dots<u_{\eta(n)}<t$, and in the last step we used the inequality $|e^{-x}-e^{-y}|\le |e^{-x}+e^{-y}||x-y|^\alpha\le 2 |x-y|^\alpha$ for $x,y>0$ and $\alpha\in(0,1]$. 

Let
\[A_t^2(u)=\int_{\R^{nd}} \left(\Psi(\xi_1+\dots+\xi_n)\right)^{\alpha_2} \prod_{j=1}^{n} e^{-2(u_{\eta(j+1)}-u_{\eta(j)})\Psi(\xi_{\eta(1)}+\dots+\xi_{\eta(j)})} \mu(d\xi)\]
and 
\[A_s^2(v)=\int_{\R^{nd}} \left(\Psi(\xi_1+\dots+\xi_n)\right)^{\alpha_2} \prod_{j=1}^{n} e^{-2(v_{\sigma(j+1)}-v_{\sigma(j)})\Psi(\xi_{\sigma(1)}+\dots+\xi_{\sigma(j)})} \mu(d\xi).\]
we have
\begin{align*}
 &\int_{[0,t]^n}A_t^2(u)du \\
 =&n!\int_{[0<u_1<\dots<u_n<t]} \int_{\R^{nd}} \left(\Psi(\xi_1+\dots+\xi_n)\right)^{\alpha_2} \prod_{j=1}^{n} e^{-2(u_{j+1}-u_{j})\Psi(\xi_{1}+\dots+\xi_{j})} \mu(d\xi)du\\
 \le& n! \sup_{z\in \R^d}\int_0^t\int_{\R^d} \left(\Psi(\xi_n+z)\right)^{\alpha_2}e^{-2(t-u_{n})\Psi(\xi_n+z)}\mu(d\xi_n)du_n\\
 &~~~~~~~~\times \int_{[0<u_1<\dots<u_{n-1}<t]}\int_{\R^{(n-1)d}} \prod_{j=1}^{n-1} e^{-2(u_{{j+1}}-u_{{j}})\Psi(\xi_{1}+\dots+\xi_{j})}\mu(d\xi)du\\
\le& n! C \int_{[0<u_1<\dots<u_{n-1}<t]}\int_{\R^{(n-1)d}} \prod_{j=1}^{n-1} e^{-2(u_{{j+1}}-u_{{j}})\Psi(\xi_{j})}\mu(d\xi)du\\
 \le& n! \frac{C^{n+1}T^n}{\Gamma(n\alpha_2+1)},
\end{align*}
where the last second step follows from Lemma \ref{lemma0}, Hypothesis (T2), Remark \ref{rm5.7} and Lemma \ref{lemma4.6}, and the last step follows by a similar argument for (\ref{e5.26'}). Now we have the estimation for $B_n$,
\begin{align}
 B_n&\le (t-s)^{\alpha_2} \frac1{(n!)^2} \int_{[0,t]^n\times[0,s]^n} (A_t^2(u) + A_s^2(v)) \prod_{j=1}^n|u_j-v_j|^{-\beta_0}dvdu\notag\\
 &\le 2(t-s)^{\alpha_2} \frac1{(n!)^2} D_t^n \int_{[0,t]^n} A_t^2(u)du\notag\\
 &\le 2(t-s)^{\alpha_2} \frac1{n!} D_t^n \frac{C^{n+1}T^n}{\Gamma(n\alpha_2+1)}.
 \label{estbn}
 \end{align}
 
 {\bf Step 2(c).} Therefore, combining (\ref{estan}) and (\ref{estbn}), we have that there exists a constant $C$ depending on the measure $\mu$, $T, \alpha_2$ and $ \beta_0$ such that \begin{equation}\label{e5.26}
\sum_{n=1}^\infty (p-1)^{n/2}\sqrt{n!}\sqrt{D_n}=\sum_{n=1}^\infty (p-1)^{n/2}\sqrt{n!} \sqrt{A_n+B_n}\le C (t-s)^{[\alpha_2\wedge (1-\beta_0)]/2}.
\end{equation}

Note that we can get  estimation for $D_n'$ analogous to (\ref{e5.26}), by  an argument similar as the above for $D_n$. Finally, by (\ref{lp}), (\ref{dn}) and (\ref{e5.26}), we have
 \begin{align*}
 &\|u(t,x)-u(s,x)\|_{L^p}\le\sum_{n=1}^\infty (p-1)^{n/2}\sqrt{n!} \|h_n(\cdot,t,x)-h_n(\cdot,s,y)\|_{\mathcal H^{\otimes n}}\notag \\
 &\le \sum_{n=1}^\infty (p-1)^{n/2}\sqrt{n!} \sqrt{D_n+D_n'}\le C (t-s)^{[\alpha_2\wedge (1-\beta_0)]/2}.
 \end{align*}
 The H\"older continuity in time now is concluded by the Kolmogorov's criterion.
\end{proof}


\bibliographystyle{abbrv}
\bibliography{SPDE-Levy}
\begin{tabular}{lll}
Jian Song \\
Department of Mathematics and Department of Statistics \& Actuarial Science\\
The University of Hong Kong, Hong Kong\\
{\tt txjsong@hku.hk}
\end{tabular}

\end{document}